\documentclass[A4]{amsart}

\usepackage[square,sort,comma,numbers]{natbib}
\usepackage{amsmath, amssymb, amstext, amsfonts, textcomp, amsxtra, amsbsy, amsgen, amsopn, amscd, mathrsfs, amsthm, latexsym, array}
\usepackage[textwidth=16cm,textheight=22cm,centering]{geometry}

\usepackage{color}

\usepackage[all]{xy}

\newtheorem{theorem}             {Theorem}  [section]
\newtheorem{definition} [theorem] {Definition}
\newtheorem{lemma}      [theorem]{Lemma}
\newtheorem{corollary}  [theorem]{Corollary}
\newtheorem{proposition}[theorem]{Proposition}
\newtheorem{remark} [theorem] {Remark}

\numberwithin{equation}{section} \everymath{\displaystyle}

\newcommand{\Cont}{{\rm C}}
\newcommand{\Aut}{\mathcal{A}}

\newcommand{\Sch}{\mathcal{S}}

\newcommand{\intL}{{\rm L}}
\newcommand{\Ht}{{\rm Ht}}

\newcommand{\hol}{{\rm hol}}
\newcommand{\Nr}{{\rm Nr}}

\newcommand{\gp}[1]{\mathbf{#1}}
\newcommand{\GL}{{\rm GL}}
\newcommand{\PGL}{{\rm PGL}}
\newcommand{\SL}{{\rm SL}}

\newcommand{\diag}{{\rm diag}}

\newcommand{\ag}[1]{\mathbb{#1}}
\newcommand{\Mat}{{\rm M}}

\newcommand{\lcd}{{\rm lcd}}

\newcommand{\F}{\mathbf{F}}

\newcommand{\vo}{\mathfrak{o}}
\newcommand{\vp}{\mathfrak{p}}
\newcommand{\idlJ}{\mathfrak{J}}
\newcommand{\Dis}{{\rm D}}
\newcommand{\Dif}{\mathfrak{D}}
\newcommand{\gCl}{{\rm Cl}}

\newcommand{\ProjP}{{\rm P}}
\newcommand{\norm}[1][\cdot]{\lvert #1 \rvert}
\newcommand{\extnorm}[1]{\left\lvert #1 \right\rvert}
\newcommand{\Norm}[1][\cdot]{\lVert #1 \rVert}
\newcommand{\extNorm}[1]{\left\lVert #1 \right\rVert}

\newcommand{\Four}[2][]{\mathfrak{F}_{#1}(#2)}

\newcommand{\Mellin}[2][]{\mathfrak{M}_{#1}(#2)}

\newcommand{\rpR}{{\rm R}}
\newcommand{\Res}{{\rm Res}}
\newcommand{\Ind}{{\rm Ind}}

\newcommand{\Intw}{\mathcal{M}}

\newcommand{\cusp}{{\rm cusp}}
\newcommand{\Eis}{{\rm E}}

\newcommand{\Cond}{\mathbf{C}}
\newcommand{\cond}{\mathfrak{c}}
\newcommand{\fin}{{\rm fin}}
\newcommand{\eis}{{\rm E}}
\newcommand{\eisCst}{{\rm E}_{\gp{N}}}
\newcommand{\Reis}{\mathcal{E}}
\newcommand{\reg}{{\rm reg}}

\newcommand{\freg}{{\rm fr}}
\newcommand{\FundD}{\mathcal{D}}
\newcommand{\Ex}{\mathcal{E}{\rm x}}

\newcommand{\latL}{\mathcal{L}}
\newcommand{\latPlg}{\mathcal{P}}

\newcommand{\fsB}{{\rm B}}
\newcommand{\fsH}{{\rm H}}

\newcommand{\Vol}{{\rm Vol}}
\makeatletter

\newcommand{\Rmnum}[1]{\expandafter\@slowromancap\romannumeral #1@}
\makeatother

\title{A Note on Spectral Analysis in automorphic representation theory for $\GL_2$: \Rmnum{2}}
\author{Han Wu}

\begin{document}

\footnote{Affiliation of Author: MTA R\'enyi Int\`ezet Lend\"ulet Automorphic Research Group}

	\begin{abstract}
		We generalize ZagierÕs work on regularized integral to the singular case in the adelic setting. We develop necessary tools of treating various singular cases of regularized triple product formulas, which appear naturally in the work of Michel and Venkatsh on the subconvexity problem for $\GL_2$.
	\end{abstract}
	
	\maketitle
	
	\tableofcontents

\section{Introduction}

	Trying to enlarge the applicability of the Rankin-Selberg method, Zagier \cite{Za82} invented the regularized integral in automorphic representation theory for $\PGL_2$, which deals with certain non convergent integrals naturally appearing in the theory. This idea was generalized and successfully applied to the problem of subconvexity for $\GL_2$ in analytic number theory by Michel and Venkatesh \cite[\S 4.3]{MV10}. Roughly speaking, three definitions of regularized integral are available:

\noindent (1) We subtract from $\varphi$ a non integral part $\Reis(\varphi)$, considered to have integral $0$ by abuse of orthogonality, and define
	$$ \int^{\reg} \varphi = \int \varphi - \Reis(\varphi). $$
	
\noindent (2) We introduced some suitable meromorphic function $E(s)$ with constant residue $1$ at $0$, such that
	$$ \int \varphi \cdot E(s) $$
is convergent for certain range of $s$ and has a meromorphic continuation. We then define
	$$ \int^{\reg} \varphi = \Res_{s=0} \int \varphi \cdot E(s). $$
	
\noindent (3) We find a/any measure-operator $\sigma_0$ such that the convolution $\sigma_0 * \varphi$ becomes integrable, and that the dual measure $\sigma_0^{\vee}$ with respect to change of variables satisfies $\sigma_0^{\vee} * 1 = R_0 \neq 0$. We then define
	$$ \int^{\reg} \varphi = R_0^{-1} \int \sigma_0 * \varphi. $$
	
\noindent The idea of (1) seems to come from physicists. (2) has a variation in a simpler case \cite[\S 4.3.2]{MV10}, and finds its incarnation in the automorphic representation theory as the Rankin-Selberg methods, which provides the concrete computability of the regularized integral. Zagier's orginal work provides the equivalence of (1) and (2). (3) is due to Michel and Venkatesh. It is an extension of the theory of regularized integral, not just another equivalent definition. (In fact, $\sigma_0 * \varphi$ is really a different way of making $\varphi$ integrable from the subtraction $\varphi - \Reis(\varphi)$, which in the application to the subconvexity problem \cite{MV10} is at least ``more suited'' \cite[\S 4.3.6]{MV10} from the strategical viewpoint.)

	However, neither \cite{Za82} nor \cite[\S 4.3]{MV10} was capable of explicitly treating an exceptional case (we shall call it the \emph{singular case}), in which $\varphi$ has the quasi-character $\norm_{\ag{A}}$ in its set of exponents \cite[\S 4.3.3]{MV10}. An example is given by the \emph{regularizing Eisenstein series} Definition \ref{RegEisDef}. It is easy to see that any measure $\sigma_0$ made from Hecke operators which makes the regularizing Eisenstein series integrable also makes $\sigma_0^{\vee}=\sigma_0$ annihilate $1$ (c.f. Remark \ref{KillRegEis}). Hence the definition (3) can not be extended to the singular case. In order to avoid the singular case, the authors of \cite{MV10} applied the technic of \emph{deformation} \cite[\S 5.2.6]{MV10}, which not only introduced complications in many auxiliary computations \cite[\S 3.1.11, 3.2.4, 3.2.8, 4.1.9 \& 4.4.3]{MV10}, but also reduced the precision of various estimations. Hence it is reasonable to look for a treatment of the singular case. (3) failing, we turn back to the original idea of Zagier, i.e., (1) and (2). For example, in the line after \cite[(18)]{Za82}, the author excluded the discussion of the possibility of ``$\alpha_i=1$'', which corresponds to the singular case, as well as the possibility of ``$\alpha_i=0$''. This is unnatural, since the case ``$\alpha_i=0$'' is in the range of integrability, being easily handled by subtracting a constant function to reduce to the case of non-existence of ``$\alpha_i=0$'' in (2) (c.f. Theorem \ref{AdelicRegThm} (4) for a rigorous treatment); while in the case of ``$\alpha_i=1$'' the formula in Theorem \ref{AdelicRegThm} (4) still makes sense, except that the pole of $R(s,\varphi)$ at $s=1/2$ may not be simple. It will be consistent with (1) if we allow the subtraction of the regularizing Eisenstein series (c.f. Definition \ref{RegEisDef}), as well as its derivatives. Moreover, in Remark \ref{KillRegEis} although $(T(\vp)-1)^2$ annihilates both $\eis^{\reg}(1/2,e_0)$ and $1$, it ``kills twice'' $1$ while $T(\vp)-1$ alone does not ``kill'' $\eis^{\reg}(1/2,e_0)$, which intuitively suggests \emph{defining}
	$$ \int_{[\PGL_2]}^{\reg} \eis^{\reg}(1/2,e_0) =0 $$
from the viewpoint of definition (3). However, such an extension of regularized integral to the singular case looses $\PGL_2(\ag{A})$-invariance (c.f. Remark \ref{RegIntGinvLoss}). All these suggest an extension to the singular case of the regularized integral, at least for the beauty of the theory itself\footnote{The analysis of this paper, together with some technics treating the finite places with ``more representation theoretic than analytic'' arguments, improves the $\epsilon$-power dependence at the finite places into logarithmic power dependence in \cite{MV10}. This will be explained in a future paper.}. The main purpose of this paper is to further develop this idea, in order to handle various regularized triple products of Eisenstein series in the singular case, which appear naturally in the work of \cite{MV10}.
	
	In Section 2, we give a treatment of the basic theory in the singular case. For completeness, we also include a full treatment in the ``regular case'' in the adelic setting, which was first developed in \cite[\S 4.3]{MV10}. But even in the regular case, we stick strictly to the original idea of Zagier, avoiding the definition (3). In particular, we establish the $\PGL_2(\ag{A})$-invariance of the regularized integral directly in Proposition \ref{RegGInv} (to be compared with \cite[\S 4.3.6]{MV10}). In Section 3, we treat various singular cases of products of Eisenstein series. The idea of ``deformation'', utilized in \cite{MV10}, turns out to be essential, since the computation directly from the definition fails to give explicit results. In Section 4, we develop more tools, which, together with the results in Section 2, allow one to treat any singular case of triple products of Eisenstein series. We also set an example by giving an explicit formula in one such case.
	
	As a fundamental preliminary, bounding the smooth Eisenstein series is necessary. This follows the same (classical) idea in the $\gp{K}$-finite case, with a little more complicated technics. We also give a detailed treatment in Appendix for completeness. We also note that the use of smooth Eisenstein series seems to be inevitable in \cite{MV10}.
	
	The regularized inner product formula, defined via the regularized integral \cite[\S 4.3.5 \& 4.3.8]{MV10}, is a natural extension of the Plancherel formula in the theory of automorphic representation for $\GL_2$. For this reason, we put this paper into the current series of work. Note that an extra ``(more) degenerate term'' need to be added into the formula \cite[(4.20)]{MV10} in the singular case, which can be determined by Theorem \ref{RIPEisSing}.

\section{Zagier's Regularized Integral}

	\subsection{Regularized Integral on $\ag{R}_+$}
	
\begin{definition}
	Let $a: \ag{R}_+ \to \ag{C}$ be a continuous function. It is \emph{regularizable} if
\begin{itemize}
	\item[(1)] for any $N \gg 1$ as $t \to \infty$
	$$ a(t) = f(t)+O(t^{-N}), \quad f(t)=\sum_{i=1}^l \frac{c_i}{n_i!} t^{\frac{1}{2}+\alpha_i} \log^{n_i}t, $$
where $c_i, \alpha_i \in \ag{C}, n_i \in \ag{N}$;
	\item[(2)] for some $\alpha \in \ag{R}$ as $t \to 0^+$
	$$ a(t) = O(t^{\alpha}). $$
 \end{itemize}
 	In this case, we write for $T > 0$
 	$$ h_T(s) = \sum_{i=1}^l \frac{c_i}{n_i!} \frac{\partial^{n_i}}{\partial s^{n_i}} \left( \frac{T^{s+\alpha_i}}{s+\alpha_i} \right) = \sum_{i=1}^l c_i \sum_{m=0}^{n_i} \frac{(-1)^{n_i-m}}{m!} \frac{T^{s+\alpha_i} \log^m T}{(s+\alpha_i)^{n_i-m+1}}. $$
 \label{RegFuncRDef}
\end{definition}
\begin{lemma}
	For $f(t)$ as above, if $\lim_{T \to \infty} \int_1^T f(t) \frac{dt}{t^2}$ exists, then $\Re \alpha_i < 1/2$ for all $1 \leq i \leq l$.
\label{IntExpBd}
\end{lemma}
\begin{proof}
	If not, the condition implies that
\begin{align*}
	\int_1^T f(t) \frac{dt}{t^2} &= \sum_{\alpha_j \neq \frac{1}{2}} c_j \sum_{m=0}^{n_j} \frac{(-1)^{n_j-m}}{(\alpha_j - \frac{1}{2})^{n_j-m+1}} T^{\alpha_j-\frac{1}{2}} (\log T)^m - \sum_{\alpha_j \neq \frac{1}{2}} c_j \frac{(-1)^{n_j}}{(\alpha_j-\frac{1}{2})^{n_j+1}} \\
	&\quad + \sum_{\alpha_j=\frac{1}{2}} \frac{c_j}{n_j!} \frac{(\log T)^{n_j+1}}{n_j+1}
\end{align*}
is bounded as $T \to +\infty$. Let $\sigma = \max_j \Re \alpha_j$. We distinguish two cases.

\noindent (1) $\sigma=1/2$. Let $l=\max \{ n_j: \Re \alpha_j = 1/2, \alpha_j \neq 1/2 \} \cup \{ n_j+1: \alpha_j=1/2 \}$. We divide both sides of the equation by $(\log T)^l$ and let $T \to +\infty$ to get
\begin{equation}
	\lim_{T \to \infty} \sum c_j T^{i\tau_j} = 0
\label{AbsEq}
\end{equation}
where $\Im \alpha_j = \tau_j$ for $j$ such that either $\Re \alpha_j = 1/2, \alpha_j \neq 1/2, n_j=l$ or $\alpha_j=1/2, n_j+1=l$. In particular $\tau_j$ are mutually distinct.

\noindent (2) $\sigma > 1/2$. Let $l=\max \{ n_j: \Re \alpha_j = \sigma \}$. We divide both sides of the equation by $T^{\sigma-1/2} (\log T)^l$ and let $T \to +\infty$ to get an equation of the same form as (\ref{AbsEq}).

\noindent We conclude because (\ref{AbsEq}) contradicts the following Corollary \ref{ErgodicLemma}.
\end{proof}
\noindent Write $\ag{T} = \ag{R} / \ag{Z}$. Let $\vec{\theta} = (\theta_1, \cdots, \theta_n) \in \ag{R}^n$. For any $\vec{x} \in \ag{R}^n$, we write by $[\vec{x}]$ its image in $\ag{T}^n$. Define
	$$ \ag{T}_{\vec{\theta}} = \overline{\lbrace [t\vec{\theta}] : t\in \ag{R} \rbrace}. $$
It is a closed subgroup of $\ag{T}^n$. Furthermore, the one parameter subgroup $U_{\vec{\theta}} = \lbrace t.\vec{\theta} : t \in \ag{R} \rbrace$ of $\ag{R}^n$ acts uniquely ergodically on $\ag{T}_{\vec{\theta}}$ w.r.t. the Haar measure of $\ag{T}_{\vec{\theta}}$. More precisely,
\begin{lemma}
	For any $f \in \Cont(\ag{T}_{\vec{\theta}})$, we have
	$$ \lim_{T \to \infty} \frac{1}{T}\int_0^T f([\vec{x}] + [t.\vec{\theta}]) dt = \int_{\ag{T}_{\vec{\theta}}} f dm_{\theta} $$
	for all $[\vec{x}] \in \ag{T}_{\vec{\theta}}$. Here $dm_{\theta}$ is the normalized Haar measure on $\ag{T}_{\vec{\theta}}$.
\end{lemma}
\begin{proof}
	Consider the group of characters ${\rm Ch}(\ag{T}^n)$ of $\ag{T}^n$ given by
	$$ e_{\vec{n}}(\vec{x}) = e(\vec{n} \cdot \vec{x}) = e(\sum_{i=1}^n n_i x_i), \vec{n} \in \ag{N}^n, \vec{x} \in \ag{R}^n $$
where $e(x) = e^{2\pi i x}$. By the duality theorem for locally compact abelian groups, the group of characters ${\rm Ch}(\ag{T}_{\vec{\theta}})$ is the quotient of ${\rm Ch}(\ag{T}^n)$ by the subgroup of $e_{\vec{n}}$'s which vanish on $\ag{T}_{\vec{\theta}}$. Obviously, we have
	$$ e_{\vec{n}}(\ag{T}_{\vec{\theta}})=1 \Leftrightarrow e(t\vec{n} \cdot \vec{\theta}) = 1 \text{ for } \forall t \in \ag{R} \Leftrightarrow \vec{n} \cdot \vec{\theta}=0. $$
	So ${\rm Ch}(\ag{T}_{\vec{\theta}})$ are $e_{\vec{n}}$'s modulo the subgroup of $e_{\vec{n}}$'s with $\vec{n} \cdot \vec{\theta}=0$. Let $[e_{\vec{n}}] \neq 0$ denote a non trivial equivalence class of $e_{\vec{n}}$ in the quotient group, we calculate
	$$ \extnorm{ \frac{1}{T} \int_0^T [e_{\vec{n}}]( [\vec{x}] + [t\vec{\theta}] ) dt } = \extnorm{ \frac{e(\vec{n} \cdot \vec{x})}{T} \cdot \frac{e(T\vec{n} \cdot \vec{\theta} ) - 1}{\vec{n} \cdot \vec{\theta}} } \leq \frac{2}{\norm[T] \cdot \norm[\vec{n} \cdot \vec{\theta}]} \to 0, T \to \infty. $$
	The lemma is thus proved for $f = [e_{\vec{n}}]$, hence the $\ag{C}$-vector space generated by ${\rm Ch}(\ag{T}_{\vec{\theta}})$, which is also a $*$-subalgebra of $\Cont(\ag{T}_{\vec{\theta}})$. The lemma then follows by a standard application of the complex version of the Stone-Weierstrass theorem.
\end{proof}
\begin{corollary}
	Consider the function $f(x) = \sum_{k=1}^n a_k x^{i \theta_k}, x \in \ag{R}_+$, where $a_k \in \ag{C}, \theta_k \in \ag{R}, 1\leq k \leq n$. If $\lim_{x \to +\infty} f(x)$ exists, then $f(x)$ is a constant function.
\label{ErgodicLemma}
\end{corollary}
\begin{proof}
	Note that $f(e^{2\pi t}) = \sum_{k=1}^n a_k e(t\theta_k), t \in \ag{R}$. If $f$ is not constant, we can find $t_1, t_2 \in \ag{R}$ s.t. $f(e^{2\pi t_1}) \neq f(e^{2\pi t_2})$. By continuity, we then find some neighborhood $U_1$ of $[t_1 \vec{\theta}]$, and $U_2$ of $[t_2 \vec{\theta}]$ such that
	$$ \sum_{k=1}^n a_k e(x_k) \neq \sum_{k=1}^n a_k e(y_k), \forall \vec{x} \in U_1, \vec{y} \in U_2. $$
	But the flow $\gamma(t) = [t\vec{\theta}]$ meets both $U_1$ and $U_2$ infinitely often by the lemma, hence $\lim_{t \to \infty} f(e^{2\pi t})$ can not exist, contradicting the hypothesis.
\end{proof}

	\subsection{Regularized Integral for $\PGL_2$}
	
\begin{definition}
	Let $\varphi: \GL_2(\F)\gp{Z}(\ag{A}) \backslash \GL_2(\ag{A}) \to \ag{C}$ be a continuous function. It is \emph{slowly increasing} if for some $c \in \ag{R}$ and $g$ lying in some Siegel domain we have
	$$ \norm[\varphi(g)] \ll \Ht(g)^c, \Ht(g) \to \infty. $$
\label{SlowIncDef}
\end{definition}
\begin{definition}
	Let $\varphi: \GL_2(\F)\gp{Z}(\ag{A}) \backslash \GL_2(\ag{A}) \to \ag{C}$ be a slowly increasing function. Its \emph{regularizing kernel} $a(t,\varphi)$ is
	$$ a(t,\varphi) = \int_{\F \backslash \ag{A} \times \F^{\times} \backslash \ag{A}^{(1)} \times \gp{K}} \varphi(n(x)a(t^+y)k) dx d^{\times}y dk, $$
	where $t^+ = s_{\F}(t)$ is the image of $t$ under the section of the adelic norm map $\ag{A}^{\times} \to \ag{R}_+$ recalled in the beginning of this paper.
\end{definition}
\begin{definition}
	We call a slowly increasing function $\varphi: \GL_2(\F)\gp{Z}(\ag{A}) \backslash \GL_2(\ag{A}) \to \ag{C}$ \emph{regularizable} if its regularizing kernel $a(t,\varphi)$ satisfies the condition (1) of Definition \ref{RegFuncRDef}. In this case, we define for $s \in \ag{C}, \Re s \gg 1$
	$$ R(s,\varphi) = \int_0^{\infty} (a(t,\varphi)-f(t)) t^{s-\frac{1}{2}} \frac{dt}{t}; \quad R^*(s,\varphi) = \Lambda_{\F}(1+2s) R(s,\varphi). $$
	The space of regularizable functions is denoted by $\Aut^{\reg}(\GL_2)$.
\label{RegFuncDef}
\end{definition}
\begin{remark}
	This is equivalent to saying $a(t,\varphi)$ regularizable, i.e., the condition (2) of Definition \ref{RegFuncRDef} is automatically satisfied due to the following Corollary \ref{SIncBdAt0}.
\end{remark}
\begin{lemma}
	If $\gamma \in \GL_2(\F) - \gp{B}(\F)$, then we have $ \Ht(\gamma g) \leq \Ht(g)^{-1}$.
\label{HtBd}
\end{lemma}
\begin{proof}
	This is \cite[Lemma 3.19]{Wu5}.
\end{proof}
\begin{corollary}
	If $\varphi$ is slowly increasing as in Definition \ref{SlowIncDef}, then we have
	$$ \norm[\varphi(g)] \ll \Ht(g)^{\min(0,-c)}, \Ht(g) \to 0. $$
\label{SIncBdAt0}
\end{corollary}
\begin{proof}
	If $c \leq 0$, then it is easy to see that $\varphi$ is bounded, since elements of bounded height in a Siegel domain form a compact subset and $\varphi$ is continuous. The same argument shows that if $c<0$ we can assume $\norm[\varphi(g)] \ll \Ht(g)^c$ to hold in a whole Siegel domain $S$ containing a fundamental domain. If $\Ht(g)$ is small, we take $\gamma \in \GL_2(\F)$ such that $\gamma g \in S$, thus by the lemma we get
	$$ \norm[\varphi(g)] = \norm[\varphi(\gamma g)] \ll \Ht(\gamma g)^c \leq \Ht(g)^{-c}. $$
\end{proof}
\noindent The following function together with its Taylor expansion plays an important role:
\begin{equation}
	\lambda_{\F}(s) := \frac{\Lambda_{\F}(-2s)}{\Lambda_{\F}(2+2s)} = \frac{\lambda_{\F}^{-1}(0)}{s} + \sum_{n=0}^{\infty} \frac{s^n}{n!} \lambda_{\F}^{(n)}(0).
\label{lambdaFDef}
\end{equation}
Recall the truncation operator $\Lambda^c$ defined in \cite[(5.5)]{GJ79}. Let $f_0 \in \Ind_{\gp{B}(\ag{A}) \cap \gp{K}}^{\gp{K}} (1,1)$ be constant equal to $1$ on $\gp{K}$.
\begin{theorem}
	(Adelic version of regularization due to Zagier \cite{Za82})
\begin{itemize}
	\item[(1)] Let $\varphi: \GL_2(\F)\gp{Z}(\ag{A}) \backslash \GL_2(\ag{A}) \to \ag{C}$ be a slowly increasing function. For $s \in \ag{C}, \Re s \gg 1$ and any $T \gg 1$ we have
	$$ \int_{[\PGL_2]} \varphi(g) \Lambda^T \Eis(s,f_0)(g) dg = \int_0^T a(t,\varphi)t^{s-\frac{1}{2}} \frac{dt}{t} - \int_T^{\infty} a(t,\varphi) t^{-s-\frac{1}{2}} \frac{dt}{t} \cdot \lambda_{\F}(s-1/2). $$
	\item[(2)] If $\varphi$ is, in addition, regularizable, then we have for $T \gg 1$
\begin{align*}
	&\quad R^*(s,\varphi) + \Lambda_{\F}(1+2s) h_T(s) + \Lambda_{\F}(1-2s) h_T(-s) \\
	&= \int_{[\PGL_2]} \varphi(g) \Lambda^T \Eis^*(s,f_0)(g) dg + \\
	&\quad \Lambda_{\F}(1+2s) \int_T^{\infty} (a(t,\varphi)-f(t)) t^{s-\frac{1}{2}} \frac{dt}{t} + \Lambda_{\F}(1-2s) \int_T^{\infty} (a(t,\varphi)-f(t)) t^{-s-\frac{1}{2}} \frac{dt}{t}.
\end{align*}
	In particular, $R(s,\varphi)$ has a meromorphic continuation to $s \in \ag{C}$ with possible poles at $s=\pm 1/2, \pm \alpha_i, (\rho-1)/2$ for $\rho$ running over the non-trivial zeros of $\zeta_{\F}$, and satisfies the functional equation
	$$ R^*(s,\varphi) = R^*(-s,\varphi). $$
	\item[(3)] Under the condition of (2), if $\Re \alpha_i < 0$ for all $1 \leq i \leq l$, then we have
	$$ R(s,\varphi) = \int_{[\PGL_2]} \varphi(g) \Eis(s,f_0)(g) dg, \quad \max_{1 \leq i \leq l} \alpha_i < \Re s < -\max_{1 \leq i \leq l} \alpha_i. $$
	\item[(4)] Under the condition of (2), if $\varphi$ is integrable on $ [\PGL_2] := \GL_2(\F)\gp{Z}(\ag{A}) \backslash \GL_2(\ag{A}) $, then we have $\Re \alpha_i < 1/2$ for all $1 \leq i \leq l$, and
	$$ \int_{[\PGL_2]} \varphi(g) dg = \frac{1}{\lambda_{\F}^{(-1)}(0)} \left( \Res_{s=\frac{1}{2}} R(s,\varphi) + c_i \delta_{\substack{\alpha_i=-\frac{1}{2} \\ n_i=0}} \right), \quad \Vol([\PGL_2]) = \lambda_{\F}^{(-1)}(0). $$
\end{itemize}
\label{AdelicRegThm}
\end{theorem}
\begin{proof}
	Since the proofs of (1) to (3) are quite similar to that in \cite{Za82}, we only mention some essential points of them. Only (4) needs more explanation.
	
\noindent (1) This is standard Rankin-Selberg unfolding together with
	$$ \Lambda^T \Eis(s,f_0)(g) = \sum_{\gamma \in \gp{B}(\F) \backslash \GL_2(\F)} f_{0,s}(\gamma g) 1_{\Ht(\gamma g) \leq T} - \sum_{\gamma \in \gp{B}(\F) \backslash \GL_2(\F)} \Intw f_{0,s}(\gamma g) 1_{\Ht(\gamma g) > T}. $$
	
\noindent (2) It follows from rewriting the two terms at the right hand side of (1). For the first term we have
\begin{align*}
	\int_0^T a(t,\varphi) t^{s-\frac{1}{2}} \frac{dt}{t} &= R(s,\varphi) - \int_T^{\infty} (a(t,\varphi)-f(t)) t^{s-\frac{1}{2}} \frac{dt}{t} + \sum_{i=1}^l \frac{c_i}{n_i!} \int_0^T t^{s+\alpha_i} \log^{n_i} t \frac{dt}{t} \\
	&= R(s,\varphi) - \int_T^{\infty} (a(t,\varphi)-f(t)) t^{s-\frac{1}{2}} \frac{dt}{t} + h_T(s).
\end{align*}
For the second term we have a similar equality.

\noindent (3) In the case $\Re \alpha_i < 0$ for all $i$, we let $T \to \infty$, taking into account Proposition \ref{GlobRDEisWhi} and $h_T(s) \to 0$, to get the asserted equation.

\noindent (4) The integrability of $\varphi$ implies $\Re \alpha_i < 1/2$ for all $i$ by Lemma \ref{IntExpBd}. We take residue at $s=1/2$ on both sides of the equation obtained in (2) and divide by $\Lambda_{\F}(1+2s)$ to see for $T \gg 1$
\begin{align*}
	&\quad \Res_{s=\frac{1}{2}} R(s,\varphi) + \Res_{s=\frac{1}{2}} h_T(s) + \lambda_{\F}^{(-1)}(0) h_T(-\frac{1}{2}) \\
	&= \lambda_{\F}^{(-1)}(0) \int_{\FundD_T} \varphi(g)dg + \Res_{s=\frac{1}{2}} \int_{\FundD-\FundD_T} \varphi(g) \left( \Eis(s,f_0)(g) - \Eis(s,f_0)_N(g) \right) dg \\
	&\quad + \lambda_{\F}^{(-1)}(0) \int_T^{\infty} (a(t,\varphi)-f(t)) \frac{dt}{t^2},
\end{align*}
	where $\FundD$ is the standard fundamental domain for $[\PGL_2]$ and $\FundD_T$ is the set of $g \in \FundD$ such that $\Ht(g) \leq T$. It is easy to see that as $T \to \infty$
	$$ h_T(-\frac{1}{2}) \to 0; \quad \int_T^{\infty} (a(t,\varphi)-f(t)) \frac{dt}{t^2} \to 0. $$
	By Proposition \ref{GlobRDEisWhi}, $\Eis(s,f_0)(g) - \Eis(s,f_0)_N(g)$ is of uniformly rapid decay with respect to $\Ht(g), g \in \FundD$ as $s$ remains in a compact neighborhood of $1/2$. Hence
	$$ \int_{\FundD-\FundD_T} \varphi(g) \left( \Eis(s,f_0)(g) - \Eis(s,f_0)_N(g) \right) dg $$
	is holomorphic at $s=1/2$. We compute $\Res_{s=\frac{1}{2}} h_T(s)$ by noting that the second summand of
	$$ \frac{d^n}{d s^n} \left( \frac{T^s}{s} \right) = \frac{(-1)^n n!}{s^{n+1}} + \int_0^{\log T} t^n e^{st} dt $$
is holomorphic at $s=0$, and get
	$$ \Res_{s=\frac{1}{2}} h_T(s) = c_i \delta_{\substack{\alpha_i=-\frac{1}{2} \\ n_i=0}}. $$
\end{proof}
\begin{definition}
	We define the \emph{regularized integral} of a regularizable function $\varphi: [\PGL_2] \to \ag{C}$ as
	$$ \int_{[\PGL_2]}^{\reg} \varphi(g) dg = \frac{1}{\lambda_{\F}^{(-1)}(0)} \left( \Res_{s=\frac{1}{2}} R(s,\varphi) + c_i \delta_{\substack{\alpha_i=-\frac{1}{2} \\ n_i=0}} \right), $$
	where $c_i, \alpha_i, n_i$ are associated with $a(t,\varphi)$ as in Definition \ref{RegFuncRDef}. We call the first term the \emph{principal part} of the regularized integral, the second the \emph{degenerate part} of the regularized integral. The regularized integral is linear and extends the integral on $\Aut^{\reg}(\GL_2) \cap \intL^1(\GL_2, 1)$.
\label{RegIntDef}
\end{definition}

	\subsection{Basic Properties}
	
\begin{definition}
	Let $\omega$ be a unitary character of $\F^{\times} \backslash \ag{A}^{\times}$. Let $\varphi$ be a smooth function on $\GL_2(\F) \backslash \GL_2(\ag{A})$ with central character $\omega$. We call $\varphi$ \emph{finitely regularizable} if there exist characters $\chi_i: \F^{\times} \backslash \ag{A}^{\times} \to \ag{C}^{(1)}$, $\alpha_i \in \ag{C}, n_i \in \ag{N}$ and smooth functions $f_i \in \Ind_{\gp{B}(\ag{A}) \cap \gp{K}}^{\gp{K}} (\chi_i, \omega \chi_i^{-1})$ for $1 \leq i \leq l$, such that for any $M \gg 1$
	$$ \varphi(n(x)a(y)k) = \varphi_{\gp{N}}^*(n(x)a(y)k) + O(\norm[y]_{\ag{A}}^{-M}), \text{ as } \norm[y]_{\ag{A}} \to \infty, $$
	where we have written the \emph{essential constant term}
	$$ \varphi_{\gp{N}}^*(n(x)a(y)k)=\varphi_{\gp{N}}^*(a(y)k)=\sum_{i=1}^l \chi_i(y) \norm[y]_{\ag{A}}^{\frac{1}{2}+\alpha_i} \log^{n_i} \norm[y]_{\ag{A}} f_i(k). $$
	In this case, we call $\Ex(\varphi)=\{ \chi_i \norm^{\frac{1}{2}+\alpha_i}: 1 \leq i \leq l \}$ the \emph{exponent set} of $\varphi$, and define
	$$ \Ex^+(\varphi) = \{ \chi_i \norm^{\frac{1}{2}+\alpha_i} \in \Ex(\varphi): \Re \alpha_i \geq 0 \}; \quad \Ex^-(\varphi) = \{ \chi_i \norm^{\frac{1}{2}+\alpha_i} \in \Ex(\varphi): \Re \alpha_i < 0 \}. $$
	The space of finitely regularizable functions with central character $\omega$ is denoted by $\Aut^{\freg}(\GL_2,\omega)$.
\label{FinRegFuncDef}
\end{definition}
\begin{remark}
	In the case $\omega=1$, a finitely regularizable is smooth and regularizable in the sense of Definition \ref{RegFuncDef}. But a smooth regularizable function doesn't need to be finitely regularizable.
\end{remark}
\begin{definition}
	In the case $\omega^{-1}\xi^2(t)=\norm[t]_{\ag{A}}^{i\mu}$ for some $\mu \in \ag{R}$, we introduce the \emph{regularizing Eisenstein series} for $f \in V_{\xi,\omega\xi^{-1}}^{\infty}$ and $s$ in a neighborhood of $(1-i\mu)/2$
	$$ \Eis^{\reg}(s,f)(g) = \Eis(s,f)(g) - \frac{\Lambda_{\F}(1-2s-i\mu)}{\Lambda_{\F}(1+2s+i\mu)} \int_{\gp{K}} f(\kappa) d\kappa \cdot \chi^{-1}(\det g) \norm[\det g]_{\ag{A}}^{\frac{i\mu}{2}}. $$
	It is holomorphic at $s=(1-i\mu)/2$.
\label{RegEisDef}
\end{definition}
\begin{remark}
	Let $e_0 \in \Res_{\gp{K}}^{\GL_2(\ag{A})} \pi(1,1)$ be the spherical function taking value $1$ on $\gp{K}$. Let $T(\vp)$ denote the order $1$ Hecke operator at a finite place $\vp$ with cardinality of the residue field $q$. For $s \neq 0$, we have
	$$ T(\vp). \eis(\frac{1}{2}+s, e_0) = \lambda_{\vp}(s) \eis(\frac{1}{2}+s, e_0), \quad \lambda_{\vp}(s) = \frac{q^{1/2+s} + q^{-(1/2+s)}}{q^{1/2}+q^{-1/2}}; $$
	$$ \text{or} \quad T(\vp). \eis^{\reg}(\frac{1}{2}+s, e_0) = \lambda_{\vp}(s) \eis^{\reg}(\frac{1}{2}+s, e_0) + (\lambda_{\vp}(s)-1) \cdot \lambda_{\F}(s). $$
	The pole of $\lambda_{\F}(s)$, defined in (\ref{lambdaFDef}), at $s=0$ is compensated by the zero of $\lambda_{\vp}(s)-1$, hence we get
	$$ T(\vp). \eis^{\reg}(\frac{1}{2}, e_0) = \eis^{\reg}(\frac{1}{2}, e_0) + \lambda_{\vp}^{(1)}(0) \cdot \lambda_{\F}^{(-1)}(0). $$
	In particular, $(T(\vp)-1)^2$ annihilates $\eis^{\reg}(1/2,e_0)$.
\label{KillRegEis}
\end{remark}
\begin{remark}{(Violation of Covariance)}
	Unlike the usual Eisenstein series, the map
	$$ \pi^{\infty}(\xi \norm^{(1-i\mu)/2}, \omega\xi^{-1} \norm^{-(1-i\mu)/2}) \to \Cont^{\infty}(\GL_2, \omega), \quad f_{(1-i\mu)/2} \mapsto \eis^{\reg}((1-i\mu)/2, f) $$
is \emph{NOT} $\GL_2(\ag{A})$-covariant. But it's still $\gp{K}$-covariant.
\end{remark}
\begin{remark}
	By Proposition \ref{GlobRDEisWhi}, \ref{GlobRDEisCst} and the definition of cusp forms, $\Aut^{\freg}(\GL_2,\omega)$ contains:
\begin{itemize}
	\item $\chi(\det g)$ for quasi-characters $\chi$ such that $\chi^2=\omega$;
	\item smooth cusp forms, i.e., $\Aut_{\cusp}^{\infty}(\GL_2,\omega)$;
	\item $\frac{\partial^n}{\partial s^n}\Eis(s,f)$ for some $n \in \ag{N}$ and smooth $f \in \Ind_{\gp{B}(\ag{A})}^{\GL_2(\ag{A})}(\xi, \omega\xi^{-1})$ with $\Re s \geq 0$ and $s \neq -i\mu/2, 1/2-i\mu/2$ for $\mu=\mu(\omega^{-1}\xi^2)$ if $\omega\xi^{-2}$ is trivial on $\ag{A}^{(1)}$;
	\item $\frac{\partial^n}{\partial s^n} \mid_{s=-i\frac{\mu}{2}} (s+i\frac{\mu}{2})\Eis^*(s,f)$ for some $n \in \ag{N}$ and smooth $f$, $\mu$ the same as above;
	\item $\frac{\partial^n}{\partial s^n}\Eis^{\reg}(\frac{1-i\mu}{2},f)$ for some $n \in \ag{N}$ and $f,\mu$ the same as above;
	\item $\varphi=\Pi_{j=1}^l \varphi_j$ where $\varphi_j \in \Aut^{\freg}(\GL_2,\omega_j)$ with $\omega = \Pi_j \omega_j$.
\end{itemize}
	In the last case, we have $ \varphi_{\gp{N}}^* = \Pi_j \varphi_{j,\gp{N}}^* $. Note that we have excluded $\Eis(s,f)$ for $\Re s < 0$. But they are actually present since they are related to the case $\Re s > 0$ by functional equation.
\end{remark}
\begin{definition}
	Let $\omega$ be a unitary character of $\F^{\times} \backslash \ag{A}^{\times}$. The $\intL^2$-\emph{residual space} of central character $\omega$, denoted by $\Reis(\GL_2,\omega)$, is the direct sum of the vector spaces $\Reis^+(\GL_2,\omega)$ resp. $\Reis^{\reg}(\GL_2,\omega)$ spanned by functions
	$$ \frac{\partial^n}{\partial s^n} \Eis(s,f), \text{if } s \neq \frac{1-i\mu}{2} \quad \text{resp.} \quad \frac{\partial^n}{\partial s^n} \Eis^{\reg}(\frac{1-i\mu}{2},f) $$
	where $s \in \ag{C}, \Re s > 0$ and for some unitary character $\xi$ of $\F^{\times} \backslash \ag{A}^{\times}$, $f \in V_{\xi,\omega\xi^{-1}}^{\infty}$ and $\mu$ as above.
\end{definition}
\noindent We have the following simple fact of which we omit the proof.
\begin{proposition}
	$\Aut^{\freg}(\GL_2,\omega)$ is stable under the right translation by $\GL_2(\ag{A})$. Moreover, for any $\varphi \in \Aut^{\freg}(\GL_2,\omega)$ and any $g \in \GL_2(\ag{A})$, their sets of exponents are the same:
	$$ \Ex(g.\varphi) = \Ex(\varphi). $$
\label{FregStab}
\end{proposition}
\begin{proof}
	Take $f \in \Res_{\gp{K}}^{\GL_2(\ag{A})} \pi(\xi, \omega\xi^{-1})$ and the flat section $f_s \in \pi(\xi \norm^s, \omega\xi^{-1} \norm^{-s})$ associated to it. It suffices to show that for any $s_0 \in \ag{C}, n \in \ag{N}$ and fixed $g \in \GL_2(\ag{A})$, the right translate by $g$ of the partial derivative of this flat section, as a function on $\GL_2(\ag{A})$
	$$ x \mapsto g.f_{s_0}^{(n)}(x) := \frac{\partial^n}{\partial s^n} \mid_{s=s_0} f_s(xg) $$
is a linear combination of such functions. The above function is clearly left invariant by $\gp{N}(\ag{A})$, of central character $\omega$. Taking $x=a(y)\kappa$ for $y \in \ag{A}^{\times}, \kappa \in \gp{K}$ and writing
	$$ \kappa g = z' n' a(y') \kappa', \quad \text{for} \quad z' \in \gp{Z}(\ag{A}), n' \in \gp{N}(\ag{A}), y' \in \ag{A}^{\times}, \kappa' \in \gp{K}, $$
where $z',n',y',\kappa'$ are viewed as functions in $\kappa$, we obtain
	$$ \frac{\partial^n}{\partial s^n} \mid_{s=s_0} f_s(a(y)\kappa g) = \sum_{k=0}^n \binom{n}{k} \xi(y) \norm[y]_{\ag{A}}^{\frac{1}{2}+s_0} (\log \norm[y]_{\ag{A}})^k \cdot \norm[y']_{\ag{A}}^{\frac{1}{2}+s_0} (\log \norm[y']_{\ag{A}})^{n-k} \omega(z') \xi(y') f(\kappa'). $$
	Although $z',n',y',\kappa'$ are not uniquely determined by $\kappa$, both $\norm[y']_{\ag{A}}$ and $\omega(z') \xi(y') f(\kappa')$ are, and define smooth functions on $\gp{K}$. Moreover, the function
	$$ f_k(\kappa) := \norm[y']_{\ag{A}}^{\frac{1}{2}+s_0} (\log \norm[y']_{\ag{A}})^{n-k} \omega(z') \xi(y') f(\kappa') \in \Res_{\gp{K}}^{\GL_2(\ag{A})} \pi(\xi, \omega\xi^{-1}). $$
	Hence we get the relation
	$$ \frac{\partial^n}{\partial s^n} \mid_{s=s_0} f_s = \sum_{k=0}^n \binom{n}{k} \frac{\partial^k}{\partial s^k} \mid_{s=s_0} f_{k,s} $$
and conclude.
\end{proof}
\begin{proposition}
	The vector space $\Reis^+(\GL_2, \omega)$ is stable under the right translation by $\GL_2(\ag{A})$.
\label{Reis+Stab}
\end{proposition}
\begin{proof}
	Take a flat section $f_s$ as in the proof of Proposition \ref{FregStab}. Fix $s_0 \in \ag{C}, n \in \ag{N}, g \in \GL_2(\ag{A})$ with $\Re s_0 > 0$. Let $e_{\vec{k}}$ be an orthonormal $\gp{K}$-isotypic basis of $\Res_{\gp{K}}^{\GL_2(\ag{A})} \pi(\xi, \omega\xi^{-1})$. There is an orthonormal $\gp{K}$-isotypic basis $\tilde{e}_{\vec{k}}$ of $\Res_{\gp{K}}^{\GL_2(\ag{A})} \pi(\omega\xi^{-1}, \xi)$ and for any $n \in \ag{N}$ there is $A > 0$ such that (c.f. \cite[\S 3]{Wu5})
	$$ \Intw_s e_{\vec{k},s} = \mu_{\vec{k}}(s) \tilde{e}_{\vec{k},-s}, \quad \extnorm{\mu_{\vec{k}}^{(n)}(s)} \ll \lambda_{\vec{k}}^A $$
for some meromorphic function $\mu_{\vec{k}}(s)$, regular in $\Re s > 0, s \neq 1/2$. Here $\lambda_{\vec{k}}$ is the eigenvalue of the Laplacian on $\gp{K}_{\infty}$ for $e_{\vec{k}}$, and the bound is uniform for $s$ lying in any  compact subset of the regular region. We also have an expansion
	$$ g.f_s = \sideset{}{_{\vec{k}}} \sum a_{\vec{k}}(s,g) e_{\vec{k},s}, $$
for some functions $a_{\vec{k}}(s,g)$ holomorphic in $s$ and smooth in $g$, since
	$$ a_{\vec{k}}(s,g) = \int_{\gp{K}} f_s(\kappa g) \overline{e_{\vec{k}}(\kappa)} d\kappa. $$
	Moreover, we have $ a_{\vec{k}}(s,g) \ll_{g,f,N} \lambda_{\vec{k}}^{-N} $ for any $N \in \ag{N}$, uniformly for $s$ lying in any compact neighborhood of $s_0$. It follows that for any $l \in \ag{N}$
	$$ a_{\vec{k}}^{(l)}(s,g) := \frac{\partial^l}{\partial s^l} \mid_{s=s_0} a_{\vec{k}}(s,g) \ll_{g,f,N} \lambda_{\vec{k}}^{-N}. $$
	We thus get
	$$ g.f_{s_0}^{(n)} = \sideset{}{_{k=0}^n} \sum \sideset{}{_{\vec{k}}} \sum \binom{n}{k} a_{\vec{k}}^{(n-k)}(s_0,g) e_{\vec{k},s_0}^{(k)}. $$
	The inner sum coincides with the value at $s=s_0$ of a flat section $f_{k,s}^{(k)}$ defined by
	$$ f_k = \binom{n}{k} \sideset{}{_{\vec{k}}} \sum a_{\vec{k}}^{(n-k)}(s_0,g) e_{\vec{k}}, $$
which is smooth. We can also verify that
\begin{align*}
	\sideset{}{_{k=0}^n} \sum \frac{\partial^k}{\partial s^k} \mid_{s=s_0} \Intw_s f_{k,s} &= \sideset{}{_{k=0}^n} \sum \sideset{}{_{\vec{k}}} \sum \binom{n}{k} a_{\vec{k}}^{(n-k)}(s_0,g) \frac{\partial^k}{\partial s^k} \mid_{s=s_0} (\mu_{\vec{k}}(s) \tilde{e}_{\vec{k},-s}) \\
	&= \sideset{}{_{\vec{k}}} \sum \frac{\partial^n}{\partial s^n} \mid_{s=s_0} (a_{\vec{k}}(s,g) \mu_{\vec{k}}(s) \tilde{e}_{\vec{k},-s}) = \frac{\partial^n}{\partial s^n} \mid_{s=s_0} \Intw_s f_s,
\end{align*}
	hence we deduce that
	$$ g.\eisCst^{(n)}(s_0,f) = \sideset{}{_{k=0}^n} \sum \eisCst^{(k)}(s_0,f_k). $$
	We conclude by
	$$ g.\eis^{(n)}(s_0,f) = \sideset{}{_{k=0}^n} \sum \eis^{(k)}(s_0,f_k) $$
since both sides are orthogonal to the cusp forms and have the same constant term.
\end{proof}
\begin{remark}
	$f_k$ constructed in both proofs of Proposition \ref{FregStab} and \ref{Reis+Stab} coincide with each other.
\end{remark}
\begin{remark}
	The subspace $\Reis^{\reg}(\GL_2,\omega)$ is stable under right translation by $\gp{K}$ but not by $\GL_2(\ag{A})$. Take $\omega = 1$ for example. Choose a finite place $\vp_0$ with uniformizer $\varpi_0$. Let $e_0 \in \Res_{\gp{K}}^{\GL_2(\ag{A})} \pi(1,1)$ be the spherical function taking value $1$ on $\gp{K}$. Let $e_1$ be defined by $e_{1,v} = e_{0,v}$ for $v \neq \vp_0$; while $e_{1,\vp_0}$ is unitary, $\gp{K}_0[\vp_0]$-invariant and orthogonal to $e_{0,\vp_0}$. For example, if $q = \Nr(\vp_0)$, we can take
	$$ e_{1,\vp_0} = \sqrt{1+q^{-1}} \cdot \left\{ \sqrt{q+1} \cdot 1_{\gp{K}_0[\vp]} - \frac{1}{\sqrt{q+1}} \cdot 1_{\gp{K}_{\vp}} \right\}. $$
	It can be computed, writing $\tilde{\lambda}_{\F}(s) = \lambda_{\F}(s-1/2)$, that
	$$ \Intw_s e_{0,s} = \tilde{\lambda}_{\F}(s) e_{0,-s}, \quad \Intw_s e_{1,s} = \tilde{\lambda}_{\F}(s) \mu_1(s) e_{0,-s}, \quad \eisCst^{\reg}(1/2,e_0) = e_{0,1/2} - \lambda_{\F}^{(-1)}(0) e_{0,-1/2}^{(1)}, $$
	$$ a(\varpi_0^{-1}).e_{0,s} = c_1(s) e_{1,s} + c_0(s) e_{0,s}; \quad \text{where} $$
	$$ \mu_1(s) = q^{-2s} \frac{1-q^{-(1-2s)}}{1+q^{-(1+2s)}}, \quad c_1(s) = \frac{q^{s+1/2} - q^{-(s+1/2)}}{q^{1/2}+q^{-1/2}}, \quad c_0(s) = \frac{q^s+q^{-s}}{q^{1/2}+q^{-1/2}}. $$
	We deduce that
	$$ \varphi := a(\varpi_0^{-1}).\eis^{\reg}(1/2,e_0) - c_0(1/2) \eis^{\reg}(1/2,e_0) - c_1(1/2) \eis^{\reg}(1/2,e_1) $$
has constant term
	$$ \varphi_{\gp{N}} = \lambda_{\F}^{(-1)}(0) (c_1^{(1)}(-1/2) - c_1(1/2) \mu_1^{(1)}(1/2)) \cdot e_{1,-1/2} + \lambda_{\F}^{(-1)}(0) c_0^{(1)}(-1/2) e_{0,-1/2} \neq 0, $$
hence $0 \neq \tilde{\varphi} \notin \Reis(\GL_2,1)$. Consequently, $a(\varpi_{\vp}^{-1}).\eis^{\reg}(1/2,e_0) \notin \Reis(\GL_2,1)$. We also deduce that
	$$ \int_{[\PGL_2]}^{\reg} a(\varpi_0^{-1}).\eis^{\reg}(1/2,e_0) = \int_{[\PGL_2]}^{\reg} \varphi = \lambda_{\F}^{(-1)}(0) c_0^{(1)}(-1/2) \neq 0, $$
hence the regularized integral is in general not $\PGL_2(\ag{A})$-invariant (c.f. Proposition \ref{RegGInv} (1) below).
\label{RegIntGinvLoss}
\end{remark}
\begin{proposition}
	Let $\varphi \in \Aut^{\freg}(\GL_2,\omega)$.
\begin{itemize}
	\item[(1)] We can always find (not unique) $\Reis(\varphi) \in \Reis(\GL_2,\omega)$ such that $\varphi - \Reis(\varphi) \in \intL^1(\GL_2, \omega)$.
	\item[(2)] If for any $\chi \in \Ex(\varphi)$, we have $\Re \chi \neq 1/2$, then there is a unique function $\Reis(\varphi) \in \Reis(\GL_2,\omega)$ such that $\varphi - \Reis(\varphi) \in \intL^2(\GL_2, \omega)$. Moreover, for any $X$ in the universal enveloping algebra of $\GL_2(\ag{A}_{\infty})$, we have $X.\varphi - X.\Reis(\varphi) \in \intL^2(\GL_2, \omega)$, i.e., $\Reis(X.\varphi) = X.\Reis(\varphi)$.
\end{itemize}
\label{ResSp}
\end{proposition}
\begin{proof}
	By Lemma \ref{IntExpBd} and Proposition \ref{GlobRDEisWhi}, it is not difficult to see that $\Reis(\GL_2,\omega) \cap \intL^2(\GL_2(\F) \backslash \GL_2(\ag{A}), \omega) = \{0\}$, which implies the uniqueness of $\Reis(\varphi)$ in (2). For the existence, we find $\chi_i, \alpha_i, n_i, f_i$ as in Definition \ref{FinRegFuncDef}. Then we take, writing $\mu_i = \mu(\omega^{-1}\chi_i^2)$,
\begin{equation}
	\Reis(\varphi) = \sum_{\substack{\Re \alpha_j > 0 \\ \alpha_j \neq \frac{1}{2}+i\mu_j}} \frac{\partial^{n_j}}{\partial s^{n_j}}\Eis(\alpha_j,\omega,\omega^{-1}\chi_j;f_j) + \sum_{\substack{\Re \alpha_j > 0 \\ \alpha_j = \frac{1}{2}+i\mu_j}} \frac{\partial^{n_j}}{\partial s^{n_j}}\Eis^{\reg}(\alpha_j,\omega,\omega^{-1}\chi_j;f_j).
\label{ReisDef}
\end{equation}
	Thus for any $\chi \in \Ex(\varphi-\Reis(\varphi))$, we have $\Re \chi \leq 1/2$ resp. $\Re \chi < 1/2$ under the condition in (2). Hence $\varphi-\Reis(\varphi) \in \intL^1(\GL_2, \omega)$ resp. $\intL^2(\GL_2, \omega)$. For the ``moreover'' part, it suffices to see that the differential operator $X$ does not increase the real part of elements in $\Ex(\varphi)$, which is essentially due to the following calculation:
	$$ y \frac{d}{dy} (y^{\sigma} \log^k y) = y^{\sigma} \log^k y + k y^{\sigma} \log^{k-1}y, \forall y>0, \sigma \in \ag{C}, k \in \ag{N}. $$
\end{proof}
\begin{definition}
	In the case (2), we call $\Reis(\varphi)$ the $\intL^2$-residue of $\varphi$. For definiteness, we shall write $\Reis(\varphi)$ to be the one given by (\ref{ReisDef}).
\end{definition}
\begin{proposition}
\begin{itemize}
	\item[(1)] For any $\Reis \in \Reis(\GL_2,1)$, we have
	$$ \int_{[\PGL_2]}^{\reg} \Reis(g) dg = 0. $$
	If moreover $\Reis \in \Reis^+(\GL_2,1)$, then for any $g_0 \in \GL_2(\ag{A})$, we have
	$$ \int_{[\PGL_2]}^{\reg} g_0.\Reis(g) dg = 0. $$
	\item[(2)] Let $\varphi \in \Aut^{\freg}(\GL_2,1)$. For any $\Reis \in \Reis(\GL_2,1)$ such that $\varphi - \Reis \in \intL^1([\PGL_2])$ we have
	$$ \int_{[\PGL_2]}^{\reg} \varphi(g) dg = \int_{[\PGL_2]} (\varphi-\Reis)(g) dg. $$
	In particular, $\int_{[\PGL_2]}^{\reg}$ is always $\gp{K}$-invariant. It is $\GL_2(\ag{A})$-invariant on the subspace of $\varphi \in \Aut^{\freg}(\GL_2,1)$ such that $\Ex(\varphi)$ does not contain $\norm_{\ag{A}}$.
\end{itemize}
\label{RegGInv}
\end{proposition}
\begin{proof}
	For (1), the second assertion follows from the first by Proposition \ref{Reis+Stab}. We calculate $a(t,\Reis)$ for $\Reis(g) = \frac{\partial^n}{\partial s^n} \Eis(s,f)(gg_0)$ with $s \neq 1/2-i\mu(\chi)$ resp. $\frac{\partial^n}{\partial s^n} \Eis^{\reg}(\frac{1}{2}-i\mu(\chi),f)(g)$ for some $f \in V_{\chi,\chi^{-1}}^{\infty}$ with $\Re s > 0$, a unitary character $\chi$ of $\F^{\times} \backslash \ag{A}^{\times}$ if $\chi \mid_{\F^{\times} \backslash \ag{A}^{(1)}} = 1$. Due to the integral $\int_{\F^{\times} \backslash \ag{A}^{(1)}} d^{\times}y$, it is easy to see that $a(t,\varphi)$ is non-vanishing only if $\chi$ is trivial on $\F^{\times} \backslash \ag{A}^{(1)}$, in which case $\mu(\omega\chi^2)=2\mu(\chi)$. We also notice that we can interchange the order of $\Intw$ and $\int_{\gp{K}} dk$ since they commute with each other.
\begin{itemize}
	\item $\Reis(g) = \frac{\partial^n}{\partial s^n} \Eis(s,f)(g)$ with $s \neq 1/2-i\mu(\chi)$: We get
	$$ a(t,\Reis) = \zeta_{\F}^*(1) \int_{\gp{K}} f(k)dk \cdot \left\{ t^{\frac{1}{2}+s+i\mu(\chi)} \log^n t + \frac{\partial^n}{\partial s^n} \left( t^{\frac{1}{2}-s-i\mu(\chi)} \frac{\Lambda_{\F}(1-2s-2i\mu(\chi))}{\Lambda_{\F}(1+2s+2i\mu(\chi))} \right) \right\}, $$
	and conclude by the fact that $a(t,\Reis)$ has no constant term as a function of $t$.
	\item $\Reis(g)=\frac{\partial^n}{\partial s^n} \Eis^{\reg}(\frac{1}{2}-i\mu(\chi),f)(g)$: We get
\begin{align*}
	a(t,\Reis) &= \zeta_{\F}^*(1) \int_{\gp{K}} f(k)dk \cdot \\
	&\quad \left\{ t \log^n t + \frac{\partial^n}{\partial s^n} \mid_{s=\frac{1}{2}-i\mu(\chi)} \left( (t^{\frac{1}{2}-s-i\mu(\chi)}-1) \frac{\Lambda_{\F}(1-2s-2i\mu(\chi))}{\Lambda_{\F}(1+2s+2i\mu(\chi))} \right) \right\} \\
	&= \zeta_{\F}^*(1) \int_{\gp{K}} f(k)dk \cdot \\
	&\quad \left\{ t \log^n t + \sum_{l=0}^n \binom{n}{l} \frac{(-1)^l}{l+1} \cdot \log^{l+1}t \cdot \frac{d^{n-l}}{d s^{n-l}} \mid_{s=0} \left( \frac{s\Lambda_{\F}(2s)}{\Lambda_{\F}(2-2s)} \right) \right\},
\end{align*}
	and conclude the same way as in the previous case.
\end{itemize}
	For (2), the first part is trivial. For the second part, we note that $\varphi-\Reis(\varphi) \in \intL^1$ implies $g_0.\varphi-g_0.\Reis(\varphi) \in \intL^1$ for any $g_0 \in \GL_2(\ag{A})$, and $g_0.\Reis(\varphi)$ has regularized integral $0$ by (1) if either $\Reis \in \Reis^+(\GL_2,1)$ or $g_0 \in \gp{K}$.
\end{proof}
\begin{remark}
	The above proof of (2) is to be compared with \cite[\S 4.3.6]{MV10}, where another simpler but indirect proof was given for the ``regular case''.
\end{remark}

\section{Product of Two Eisensetein Series: Singular Cases}

	\subsection{Deformation Technics}
	
	Above all, we have the following result in the regular case (c.f. \cite[\S 3]{Za82}).
\begin{lemma}
	Let $\xi_j,\xi_j'$ be (unitary) Hecke characters with $\xi_1\xi_1' \xi_2 \xi_2' = 1$ and write $\pi_j = \pi(\xi_j,\xi_j')$, $j=1,2$. For $f_j \in \pi_j$, we shall write $\eis^{\sharp}$ for $\eis$ or $\eis^{\reg}$, whichever is regular at $s=1/2$. If $\pi_1 \not \simeq \widetilde{\pi_2}$, then for any $n_1, n_2 \in \ag{N}$, we have
	$$ \int_{[\PGL_2]}^{\reg} \eis^{(n_1)}(0,f_1) \eis^{(n_2)}(0,f_2) = 0, \quad \int_{[\PGL_2]}^{\reg} \eis^{\sharp,(n_1)}(1/2,f_1) \eis^{\sharp,(n_2)}(1/2,f_2) = 0. $$
\label{SimpleProd}
\end{lemma}
\begin{proof}
	$R(s, \eis^{(n_1)}(0,f_1) \eis^{(n_2)}(0,f_2))$ resp. $R(s, \eis^{\sharp,(n_1)}(1/2,f_1) \eis^{\sharp,(n_2)}(1/2,f_2))$ represents the value at $s_1=s_2=0$ of $\partial_1^{n_1} \partial_2^{n_2}$ of
	$$ \Lambda(1/2+s+s_1+s_2, \xi_1\xi_2) \Lambda(1/2+s+s_1-s_2, \xi_1\xi_2') \Lambda(1/2+s-s_1+s_2, \xi_1'\xi_2) \Lambda(1/2+s-s_1-s_2, \xi_1'\xi_2') \quad \text{resp.} $$
	$$ \Lambda(3/2+s+s_1+s_2, \xi_1\xi_2) \Lambda(1/2+s+s_1-s_2, \xi_1\xi_2') \Lambda(1/2+s-s_1+s_2, \xi_1'\xi_2) \Lambda(-1/2+s-s_1-s_2, \xi_1'\xi_2'), $$
which is regular at $s=1/2$ by assumption. The degenerate part is also easily seen to be $0$ by assumption. We conclude by Definition \ref{RegIntDef}.
\end{proof}

	In other cases beyond the above one, it seems to be difficult to obtain simple formulas by definition. However, the idea of deformation does provide simple and useful formulas. In general, if $\varphi \in \Aut^{\freg}(\PGL_2), \Reis \in \Reis(\PGL_2)$ are given, so that $\varphi - \Reis \in \intL^1([\PGL_2])$, and if we can find continuous families $\varphi_s \in \Aut^{\freg}(\PGL_2), \Reis_s \in \Reis(\PGL_2)$ which coincide with $\varphi, \Reis$ at $s=0$, then we have
\begin{equation}
	\int_{[\PGL_2]}^{\reg} \varphi = \int_{[\PGL_2]} \varphi - \Reis = \lim_{s \to 0} \int_{[\PGL_2]} \varphi_s - \Reis_s = \lim_{s \to 0} \left( \int_{[\PGL_2]}^{\reg} \varphi_s - \int_{[\PGL_2]}^{\reg} \Reis_s \right).
\label{DeformTec}
\end{equation}
	All the formulas we are going to obtain will follow this principle together with (suitable simple variants of Lemma \ref{SimpleProd}). Since the computation is long, we shall only give detail in the most complicated cases. The notations in Lemma \ref{SimpleProd} will be used unless otherwise explicitly reset.

	\subsection{Unitary Series}

\begin{definition}
	If $f \in \Res_{\gp{K}}^{\GL_2(\ag{A})} \pi(\xi_1,\xi_2)$, we define for any $s \in \ag{C}$ an operator $\Intw_s: \Res_{\gp{K}}^{\GL_2(\ag{A})} \pi(\xi_1,\xi_2) \to \Res_{\gp{K}}^{\GL_2(\ag{A})} \pi(\xi_2,\xi_1)$ (abuse of notations) by requiring
	$$ \Intw_s f_s(a(y)\kappa) = \xi_2(y)\norm[y]_{\ag{A}}^{\frac{1}{2}-s} \Intw_sf(\kappa), \quad \text{i.e.,} \quad \Intw_s f_s = (\Intw_s f)_{-s}; $$
	$$ \text{resp.} \quad \widetilde{\Intw}_s = \Intw_s \circ (I-\ProjP_{\gp{K}}e_{\xi}), \quad \text{if } \xi_1=\xi_2=\xi $$
	with the Taylor expansion at $s=0$ resp. $s=1/2$ when $\xi_1=\xi_2 = \xi$ (since $\Intw_s$ is ``diagonalizable'')
	$$ \Intw_sf = \sum_{n=0}^{\infty} \frac{s^n}{n!} \Intw_0^{(n)}f, \quad \text{resp.} \quad \widetilde{\Intw}_{1/2+s} f = \sum_{n=0}^{\infty} \frac{s^n}{n!} \widetilde{\Intw}_{1/2}^{(n)}f. $$
Here $\ProjP_{\gp{K}}$, with $d\kappa$ the probability Haar measure on $\gp{K}$, is defined to be the map
	$$ \pi(\xi,\xi) \to \ag{C}, f \mapsto \int_{\gp{K}} f(\kappa) \xi^{-1}(\kappa) d\kappa, $$
and $e_{\xi} = \xi \circ \det \in \Res_{\gp{K}}^{\GL_2(\ag{A})} \pi(\xi,\xi)$.
\end{definition}
\begin{lemma}
	Let $f_1,f_2 \in \Res_{\gp{K}}^{\GL_2(\ag{A})} \pi(1,1)$. For $0 \neq s \in \ag{C}$ small, we have
	$$ \int_{[\PGL_2]}^{\reg} \eis(s,f_1) \eis^{(1)}(0,f_2) = 0. $$
\label{SimpleProdU}
\end{lemma}
\begin{proof}
	This is a variant of Lemma \ref{SimpleProd}.
\end{proof}

	We continue to use the notations in the previous lemma. We can write
	$$ \eisCst(s,f_1) \eisCst^{(1)}(0,f_2) = 2(f_1f_2)_{1/2+s}^{(1)} + 2(\Intw_s f_1f_2)_{1/2-s}^{(1)} + (f_1 \Intw_0^{(1)}f_2)_{1/2+s} + (\Intw_s f_1 \Intw_0^{(1)}f_2)_{1/2-s}. $$
	We tentatively define
\begin{align*}
	\Reis^{\reg}(s) &:= s^{-1} \left\{ 2 \eis^{\reg,(1)}(1/2+s, f_1f_2) + 2 \eis^{\reg,(1)}(1/2-s, \Intw_s f_1f_2) \right. \\
	 &\left. + \eis^{\reg}(1/2+s, f_1 \Intw_0^{(1)}f_2) + \eis^{\reg}(1/2-s, \Intw_s f_1 \Intw_0^{(1)}f_2) \right\}
\end{align*}
	Applying Lemma \ref{SimpleProdU}, \ref{SimpleProdSing} with $n=0, 1$ together with (\ref{DeformTec}), we get
\begin{align*}
	\int_{[\PGL_2]}^{\reg} \eis^{(1)}(0,f_1) \eis^{(1)}(0,f_2) &= \lim_{s \to 0} \int_{[\PGL_2]} s^{-1} \eis(s,f_1) \eis^{(1)}(0,f_2) - \Reis^{\reg}(s) \\
	&= \frac{1}{s} \cdot \left\{ 2 \frac{\lambda_{\F}^{(1)}(s)}{\lambda_{\F}^{(-1)}(0)} \ProjP_{\gp{K}}(f_1f_2) + 2  \frac{\lambda_{\F}^{(1)}(-s)}{\lambda_{\F}^{(-1)}(0)} \ProjP_{\gp{K}}(\Intw_s f_1f_2) \right. \\
	&\left. + \frac{\lambda_{\F}(s)}{\lambda_{\F}^{(-1)}(0)} \ProjP_{\gp{K}}(f_1 \Intw_0^{(1)}f_2) + \frac{\lambda_{\F}(-s)}{\lambda_{\F}^{(-1)}(0)} \ProjP_{\gp{K}}(\Intw_s f_1 \Intw_0^{(1)}f_2) \right\}.
\end{align*}
	Taking Laurent expansions, we verify that the function in $s$ in the range of the above limit is regular at $s=0$, unlike its appearance. The properties
	$$ \ProjP_{\gp{K}}(f_1 \Intw_0^{(k)}f_2) = \ProjP_{\gp{K}}(f_2 \Intw_0^{(k)}f_1), \forall k \in \ag{N}; \quad \Intw_0^{(2)} = \Intw_0^{(1)} \circ \Intw_0^{(1)} $$
	$$ \text{coming from} \quad \Intw_s \circ \Intw_{-s} = 1 $$
must be used. Taking limit as $s \to 0$, we obtain (2) of the following:
\begin{theorem}	
	If $\pi(\xi_1,\xi_2)$ is spherical, we also write 
	The regularized integral of the product of two unitary Eisenstein series is computed as:
\begin{itemize}
	\item[(1)] If $\pi_1 = \pi(\xi_1,\xi_2), \pi_2 = \pi(\xi_1^{-1}, \xi_2^{-1})$ resp. $\pi_2 = \pi(\xi_2^{-1}, \xi_1^{-1})$ and $\xi_1 \neq \xi_2$, then
	$$ \int_{[\PGL_2]}^{\reg} \eis(0,f_1) \eis(0,f_2) = \frac{2\lambda_{\F}^{(0)}(0)}{\lambda_{\F}^{(-1)}(0)} \ProjP_{\gp{K}}(f_1f_2) - \ProjP_{\gp{K}}(\Intw_0^{(1)}f_1 \cdot \Intw_0 f_2), \quad \text{resp.} $$
	$$ \frac{\lambda_{\F}^{(0)}(0)}{\lambda_{\F}^{(-1)}(0)} (\ProjP_{\gp{K}}(f_1 \Intw_0 f_2) + \ProjP_{\gp{K}}(f_2 \Intw_0 f_1)) - \ProjP_{\gp{K}}(\Intw_0^{(1)}f_1 \cdot f_2). $$
	\item[(2)] If $\pi_1 = \pi(\xi,\xi), \pi_2 = \pi(\xi^{-1},\xi^{-1})$, then
\begin{align*}
	&\quad \int_{[\PGL_2]}^{\reg} \eis^{(1)}(0,f_1) \eis^{(1)}(0,f_2) = \frac{4\lambda_{\F}^{(2)}(0)}{\lambda_{\F}^{-1}(0)} \ProjP_{\gp{K}}(f_1f_2) + \frac{4\lambda_{\F}^{(2)}(0)}{\lambda_{\F}^{-1}(0)} \ProjP_{\gp{K}}(f_1 \cdot \Intw_0^{(1)} f_2) \\
	&\quad + \frac{\lambda_{\F}^{(0)}(0)}{\lambda_{\F}^{-1}(0)} \ProjP_{\gp{K}}(\Intw_0^{(1)}f_1 \cdot \Intw_0^{(1)}f_2) - \frac{1}{3} \ProjP_{\gp{K}}(\Intw_0^{(3)}f_1 \cdot f_2) - \ProjP_{\gp{K}}(\Intw_0^{(2)}f_1 \cdot \Intw_0^{(1)}f_2).
\end{align*}
\end{itemize}
\label{RIPEisUnitary}
\end{theorem}
\begin{remark}
	It is possible to get formulas for all derivatives, exploiting more the relation $\Intw_s \circ \Intw_{-s} = 1$. Since we don't have applications of these formulas, we do not include them here.
\end{remark}

	\subsection{Singular Series}

\begin{lemma}
	Let $f, f_1, f_2 \in \Res_{\gp{K}}^{\GL_2(\ag{A})} \pi(1,1)$. For $0 \neq s \in \ag{C}$ small, we have for any $n, n_1, n_2 \in \ag{N}$
	$$ \int_{[\PGL_2]}^{\reg} \eis^{\reg, (n)}(\frac{1}{2}+s, f) = -\frac{\lambda_{\F}^{(n)}(s)}{\lambda_{\F}^{(-1)}(0)} \ProjP_{\gp{K}}(f); $$
	$$ \int_{[\PGL_2]}^{\reg} \eis^{\reg, (n_1)}(\frac{1}{2}+s, f_1) \eis^{\reg, (n_2)}(\frac{1}{2}, f_2) = 0. $$
\label{SimpleProdSing}
\end{lemma}
\begin{proof}
	The first formula follows immediately from Proposition \ref{RegGInv} and definition. The second one is a variant of Lemma \ref{SimpleProd}.
\end{proof}

	We continue to use the notations in the previous section and lemma. Denote $e=e_1$. We can write
	$$ \eisCst^{\reg}(s,f) = f_s + (\widetilde{\Intw}_s f)_{-s} + \lambda_{\F}(s-\frac{1}{2}) \ProjP_{\gp{K}}(f) \left( e_{-s} - e_{-1/2} \right); $$
\begin{align*}
	\eisCst^{\reg,(n)}(\frac{1}{2},f) &= f_{1/2}^{(n)} + \sum_{k=0}^n \binom{n}{k} (-1)^k (\widetilde{\Intw}_{1/2}^{(n-k)} f)_{-1/2}^{(k)} \\
	&+ \ProjP_{\gp{K}}(f) \cdot \left\{ \frac{(-1)^{n+1} \lambda_{\F}^{-1)}(0)}{n+1} e_{-1/2}^{(n+1)} + \sum_{k=1}^n \binom{n}{k} (-1)^k \lambda_{\F}^{(n-k)}(0) e_{-1/2}^k \right\},
\end{align*}
from which one easily deduce $\eisCst^{\reg}(1/2+s,f_1) \eisCst^{\reg,(n_2)}(1/2,f_2)$. We tentatively define
\begin{align*}
	\Reis^{\reg}(s) &:= \eis^{(n_2)}(\frac{3}{2}+s, f_1f_2) + \sum_{k=0}^{n_2} \binom{n_2}{k} (-1)^k \eis^{\reg, (k)}(\frac{1}{2}+s, f_1 \widetilde{\Intw}_{1/2}^{(n_2-k)} f_2) \\
	&\quad + \ProjP_{\gp{K}}(f_2) \cdot \left\{ \frac{(-1)^{n_2+1} \lambda_{\F}^{(-1)}(0)}{n_2+1} \eis^{\reg,(n_2+1)}(\frac{1}{2}+s, f_1) \right. \\
	&\quad \left. + \sum_{k=1}^{n_2} \binom{n_2}{k} (-1)^k \lambda_{\F}^{(n_2-k)}(0) \eis^{\reg,(k)}(\frac{1}{2}+s, f_1) \right\} \\
	&\quad + \eis^{\reg,(n_2)}(\frac{1}{2}-s, f_2 \widetilde{\Intw}_{1/2+s}f_1) \\
	&\quad + \lambda_{\F}(s) \ProjP_{\gp{K}}(f_1) \cdot \left\{ \eis^{\reg,(n_2)}(\frac{1}{2}-s, f_2) - \eis^{\reg,(n_2)}(\frac{1}{2}, f_2) \right\}.
\end{align*}
	Applying Lemma \ref{SimpleProdSing} with $n_1=0$ together with (\ref{DeformTec}), we get
\begin{align*}
	&\quad \int_{[\PGL_2]}^{\reg} \eis^{\reg}(\frac{1}{2},f_1) \eis^{\reg,(n_2)}(\frac{1}{2},f_2) = \lim_{s \to 0} \int_{[\PGL_2]} \eis^{\reg}(\frac{1}{2}+s,f_1) \eis^{\reg,(n_2)}(\frac{1}{2},f_2) - \Reis^{\reg}(s) \\
	&= \lim_{s \to 0} \sum_{k=0}^{n_2} \binom{n_2}{k} \frac{(-1)^k}{\lambda_{\F}^{(-1)}(0)} \lambda_{\F}^{(k)}(s) \ProjP_{\gp{K}}(f_1 \widetilde{\Intw}_{1/2}^{(n_2-k)} f_2) + \frac{\lambda_{\F}^{(n_2)}(-s)}{\lambda_{\F}^{(-1)}(0)} \ProjP_{\gp{K}}(f_2 \widetilde{\Intw}_{1/2+s}f_1) \\
	&\quad + \ProjP_{\gp{K}}(f_1) \ProjP_{\gp{K}}(f_2) \cdot \left\{ \frac{(-1)^{n_2+1} \lambda_{\F}^{(n_2+1)}(s)}{n_2+1} + \sum_{k=1}^{n_2} \binom{n_2}{k} (-1)^k \frac{\lambda_{\F}^{(n_2-k)}(0) \lambda_{\F}^{(k)}(s)}{\lambda_{\F}^{(-1)}(0)} + \lambda_{\F}(s) \lambda_{\F}^{(n_2)}(-s) \right\}.
\end{align*}
	Taking Laurent expansions, we verify that the function in $s$ in the range of the above limit is regular at $s=0$, unlike its appearance. The symmetry
	$$ \ProjP_{\gp{K}}(f_1 \widetilde{\Intw}_{1/2}^{(k)}f_2) = \ProjP_{\gp{K}}(f_2 \widetilde{\Intw}_{1/2}^{(k)}f_1), \forall k \in \ag{N} $$
must be used. Moreover, it can be differentiated $n_1$ times to deduce (3) of the following:

\begin{theorem}
\begin{itemize}
	\item[(1)] If $\pi_1 \not\simeq \widetilde{\pi}_2$ and $\xi_1 = \xi_1', \xi_2 \neq \xi_2'$ resp. $\xi_1 \neq \xi_1', \xi_2 = \xi_2'$ resp. $\xi_1=\xi_1', \xi_2 = \xi_2'$, $\xi_1\xi_2 \neq 1$ and $\xi_1^2\xi_2^2 = 1$, then for any $n_1,n_2 \in \ag{N}$
	$$ \int_{[\PGL_2]}^{\reg} \eis^{\reg,(n_1)}(\frac{1}{2},f_1) \cdot \eis^{(n_2)}(\frac{1}{2},f_2) = 0 \quad \text{resp.} \quad \int_{[\PGL_2]}^{\reg} \eis^{(n_1)}(\frac{1}{2},f_1) \cdot \eis^{\reg,(n_2)}(\frac{1}{2},f_2) = 0 $$
	$$ \text{resp.} \quad  \int_{[\PGL_2]}^{\reg} \eis^{\reg,(n_1)}(\frac{1}{2},f_1) \cdot \eis^{\reg,(n_2)}(\frac{1}{2},f_2) = 0. $$
	\item[(2)] If $\pi_1 = \pi(\xi_1,\xi_2), \pi_2 = \pi(\xi_1^{-1}, \xi_2^{-1})$ resp. $\pi_2 = \pi(\xi_2^{-1}, \xi_1^{-1})$ with $\xi_1 \neq \xi_2$, then for any $n_1,n_2 \in \ag{N}$
	$$ \int_{[\PGL_2]}^{\reg} \eis^{(n_1)}(\frac{1}{2},f_1) \cdot \eis^{(n_2)}(\frac{1}{2},f_2) = 0, \quad \text{resp.} $$
is a linear combination with coefficients depending only on $n_1,n_2$ and $\lambda_{\F}(s)$ of
	$$ \ProjP_{\gp{K}}(\Intw_{1/2}^{(n_1+n_2+1)}f_1 \cdot f_2); \quad \ProjP_{\gp{K}}(\Intw_{1/2}^{(l)}f_1 \cdot f_2) = \ProjP_{\gp{K}}(f_1 \cdot \Intw_{1/2}^{(l)} f_2), 0 \leq l \leq \max(n_1,n_2). $$
	\item[(3)] If $\pi_1 = \pi(\xi,\xi), \pi_2 = \pi(\xi^{-1},\xi^{-1})$, then for any $n_1,n_2 \in \ag{N}$
	$$ \int_{[\PGL_2]}^{\reg} \eis^{\reg,(n_1)}(\frac{1}{2},f_1) \cdot \eis^{\reg,(n_2)}(\frac{1}{2},f_2) $$
is a linear combination with coefficients depending only on $n_1,n_2$ and $\lambda_{\F}(s)$ of
	$$ \ProjP_{\gp{K}}(\widetilde{\Intw}_{1/2}^{(l)}f_1 \cdot f_2) = \ProjP_{\gp{K}}(f_1 \cdot \widetilde{\Intw}_{1/2}^{(l)} f_2), 0 \leq l \leq \max(n_1,n_2); $$
	$$ \ProjP_{\gp{K}}(\widetilde{\Intw}_{1/2}^{(n_1+n_2+1)}f_1 \cdot f_2); \quad \ProjP_{\gp{K}}(f_1) \ProjP_{\gp{K}}(f_2). $$
\end{itemize}
\label{RIPEisSing}
\end{theorem}

\section{Towards Singular Triple Product of Eisenstein Series}

	\subsection{Some Complement of Regularized Integral}
	
	Let $\xi_1,\xi_2, \omega$ be Hecke characters with $\xi_1 \xi_2 \omega =1$. Let $f \in \pi(\xi_1,\xi_2)$ and $\varphi \in \Cont^{\infty}(\GL_2, \omega)$, i.e., a smooth function on $\GL_2(\F) \backslash \GL_2(\ag{A})$ with central character $\omega$. Suppose $\varphi$ is \emph{finitely regularizable} defined in Definition \ref{FinRegFuncDef}.
\begin{proposition}
	For $\Re s \gg 1$ sufficiently large,
	$$ R(s,\varphi; f) := \int_{\F^{\times} \backslash \ag{A}^{\times}} \int_{\gp{K}} (\varphi_{\gp{N}} - \varphi_{\gp{N}}^*)(a(y)\kappa) f(\kappa) \xi_1(y) \norm[y]_{\ag{A}}^{s-\frac{1}{2}} d\kappa d^{\times}y $$
is absolutely convergent. It has a meromorphic continuation to $s \in \ag{C}$. If in addition
	$$ \Theta := \max_j \{ \Re \alpha_j \} < 0, $$
then we have, with the right hand side absolutely converging
	$$ R(s,\varphi; f) = \int_{[\PGL_2]} \varphi \cdot \eis(s,f), \quad \Theta < \Re s < -\Theta.  $$
	In the above region, the possible poles of $R(s, \varphi; f)$ are
\begin{itemize}
	\item $1/2 + i\mu(\xi_1\xi_2^{-1})$ if $\xi_1 \xi_2^{-1}$ is trivial on $\ag{A}^{(1)}$;
	\item $(\rho - 1)/2$ where $\rho$ runs over the non-trivial zeros of $L(s,\xi_1\xi_2^{-1})$.
\end{itemize}
	In particular $R(s,\varphi; f)$ is holomorphic for $0 \leq \Re s < \min(-\Theta,1/2)$.
\label{VRI}
\end{proposition}
\begin{proof}
	The proof is quite similar to that of Theorem \ref{AdelicRegThm} (3), except that $\Intw f_s$ is no longer explicitly computable. In fact, we have for $T > 1, \Re s \gg 1$, using the standard Rankin-Selberg unfolding
\begin{align*}
	\int_{[\PGL_2]} \varphi \cdot \Lambda^T \eis(s,f) &= R(s,\varphi; f) \\
	&- \int_{\F^{\times} \backslash \ag{A}^{\times}} \left( \int_{\gp{K}} (\varphi_{\gp{N}} - \varphi_{\gp{N}}^*)(a(y)\kappa) f(\kappa) d\kappa \right) \xi_1(y) \norm[y]_{\ag{A}}^{s-1/2} 1_{\norm[y]_{\ag{A}} > T} d^{\times} y \\
	&- \int_{\F^{\times} \backslash \ag{A}^{\times}} \left( \int_{\gp{K}} (\varphi_{\gp{N}} - \varphi_{\gp{N}}^*)(a(y)\kappa) \Intw f_s(\kappa) d\kappa \right) \xi_2(y) \norm[y]_{\ag{A}}^{-s-1/2} 1_{\norm[y]_{\ag{A}} > T} d^{\times} y \\
	&+ \Vol(\F^{\times} \backslash \ag{A}^{(1)}) \left( \sum_{j=1}^l \int_{\gp{K}} f_j(\kappa) f(\kappa) d\kappa \cdot 1_{\chi_j \xi_1(\ag{A}^{(1)}) = 1} \cdot \frac{1}{n_j !} \frac{\partial^{n_j}}{\partial s^{n_j}} \left( \frac{T^{s+\alpha_j + i\mu_j}}{s+\alpha_j + i\mu_j} \right) \right. \\
	&- \left. \sum_{j=1}^l \int_{\gp{K}} f_j(\kappa) \Intw f_s(\kappa) d\kappa \cdot 1_{\chi_j \xi_2(\ag{A}^{(1)}) = 1} \cdot \frac{(-1)^{n_j}}{n_j !} \frac{\partial^{n_j}}{\partial s^{n_j}} \left( \frac{T^{-s+\alpha_j + i\mu_j'}}{-s+\alpha_j + i\mu_j'} \right) \right),
\end{align*}
where $\mu_j$ resp. $\mu_j'$ is such that
	$$ \chi_j \xi_1(y) = \norm[y]_{\ag{A}}^{i\mu_j}, \quad \text{resp.} \quad \chi_j \xi_2(y) = \norm[y]_{\ag{A}}^{i\mu_j'}. $$
	We conclude by first shifting $s$ to the desired region, then letting $T \to \infty$. The possible poles are encoded in the possible poles of $\Intw f_s$, which are included in those of $L(1+2s, \xi_1\xi_2^{-1})^{-1}$ in the above region (c.f. for example \cite[Corollary 3.7, 3.10 \& Lemma 3.18]{Wu5}).
\end{proof}
\begin{proposition}
	Let notations be as in the previous proposition with $\Theta \leq -1/2$. Recall
	$$ \varphi_{\gp{N}}^*(n(x)a(y)k)=\varphi_{\gp{N}}^*(a(y)k)=\sum_{j=1}^l \chi_j(y) \norm[y]_{\ag{A}}^{\frac{1}{2}+\alpha_j} \log^{n_i} \norm[y]_{\ag{A}} f_j(k). $$
\begin{itemize}
	\item[(1)] If $\xi_1 \neq \xi_2$, then
	$$ \left( \frac{\partial^n R}{\partial s^n} \right)^{\hol}(\frac{1}{2},\varphi;f) = \int_{[\PGL_2]}^{\reg} \varphi \cdot \eis^{(n)}(\frac{1}{2},f) - \sideset{}{_j'}\sum \frac{\lambda_{\F}^{(n+n_j)}(0)}{\lambda_{\F}^{(-1)}(0)} \ProjP_{\gp{K}}(f_j f), $$
	where the summation is over $j$ such that $\xi_1 \chi_j (\ag{A}^{(1)}) =1, \alpha_j + i \mu(\xi_1\chi_j) = -1/2$.
	\item[(2)] If $\xi_1 = \xi_2 = \xi$, then
\begin{align*}
	\left( \frac{\partial^n R}{\partial s^n} \right)^{\hol}(\frac{1}{2},\varphi;f) &= \int_{[\PGL_2]}^{\reg} \varphi \cdot \eis^{\reg,(n)}(\frac{1}{2},f) - \sideset{}{_j'}\sum \frac{\lambda_{\F}^{(n+n_j)}(0)}{\lambda_{\F}^{(-1)}(0)}  \ProjP_{\gp{K}}(f_j f) \\
	&+ \lambda_{\F}^{(n)}(0) \cdot \ProjP_{\gp{K}}(f \cdot (\xi^{-1} \circ \det)) \cdot \int_{[\PGL_2]} \varphi \cdot (\xi \circ \det),
\end{align*}
	where the summation over $j$ is as in the previous case.
\end{itemize}
\label{TripleToDouble}
\end{proposition}
\begin{proof}
	The case (1) being simpler, we only give details for (2). By twisting, we may assume $\xi=1$. Let $s$ be small with $\Re s < 0$. The $\intL^2$-residue of $\varphi \cdot \eis(1/2+s,f)$ is given by
	$$ \Reis(s) := \sideset{}{_j} \sum \eis^{(n_j)}(s+1+\alpha_j, f_j f), $$
	where the summation is over $j$ such that $\Re \alpha_j > -1$. Define
	$$ \Reis^{\reg}(s) := \sideset{}{_j'} \sum \eis^{\reg,(n_j)}(s+1+\alpha_j, f_j f) + \sideset{}{_j^*} \sum \eis^{(n_j)}(s+1+\alpha_j, f_j f) $$
where $\sideset{}{_j'} \sum$ is the summation as in the statement and $\sideset{}{_j^*} \sum$ is the rest. By the previous proposition, we have
\begin{align*}
	R(\frac{1}{2}+s,\varphi;f) &= \int_{[\PGL_2]} \varphi \cdot \eis(\frac{1}{2}+s,f) = \int_{[\PGL_2]} \varphi \cdot \eis^{\reg}(\frac{1}{2}+s,f) + \lambda_{\F}(s) \ProjP_{\gp{K}}(f) \cdot \int_{[\PGL_2]} \varphi \\
	&= \int_{[\PGL_2]} \left( \varphi \cdot \eis^{\reg}(\frac{1}{2}+s,f) - \Reis^{\reg}(s) \right) - \int_{[\PGL_2]} \left( \Reis(s) - \Reis^{\reg}(s) \right) \\
	&\quad + \lambda_{\F}(s) \ProjP_{\gp{K}}(f) \cdot \int_{[\PGL_2]} \varphi.
\end{align*}
	Since $\Reis^{\reg,(n)}(s)$ is the $\intL^2$-residue of $\varphi \cdot \eis^{\reg,(n)}(\frac{1}{2}+s,f)$, we can compare the holomorphic part of both sides and conclude by
	$$ \Reis(s) - \Reis^{\reg}(s) = \sideset{}{_j'} \sum \lambda_{\F}^{(n_j)}(s) \ProjP_{\gp{K}}(f_jf). $$
\end{proof}

	\subsection{A Triple Product Formula}
	
\begin{theorem}
	Let $f_j \in \pi(1,1), j=1,2,3$. Then for any $n \in \ag{N}$
	$$ \int_{[\PGL_2]}^{\reg} \eis^*(0,f_1) \cdot \eis^*(0,f_2) \cdot \eis^{\reg,(n)}(\frac{1}{2},f_3) $$
is the sum of
	$$ \left( \frac{\partial^n R}{\partial s^n} \right)^{\hol}(\frac{1}{2}, \eis^*(0,f_1) \cdot \eis^*(0,f_2); f_3) $$
and a weighted sum with coefficients depending only on $\lambda_{\F}(s)$ of
	$$ \ProjP_{\gp{K}}(\Intw_0^{(l)}f_1 \cdot f_2) \ProjP_{\gp{K}}(f_3), \quad 0 \leq l \leq 3; $$
	$$ \ProjP_{\gp{K}}(f_1\cdot f_2 \cdot \widetilde{\Intw}_{1/2}^{(l)} f_3), \quad 0 \leq l \leq \max(2,n) \quad \& \quad l = n+3; $$
	$$ \ProjP_{\gp{K}}((f_1 \Intw_0 f_2 + f_2 \Intw_0 f_1) \cdot \widetilde{\Intw}_{1/2}^{(l)} f_3), \quad 0 \leq l \leq \max(1,n) \quad \& \quad l = n+2; $$
	$$ \ProjP_{\gp{K}}(\Intw_0 f_1\cdot \Intw_0 f_2 \cdot \widetilde{\Intw}_{1/2}^{(l)} f_3), \quad 0 \leq l \leq n \quad \& \quad l = n+1. $$
\end{theorem}
\begin{proof}
	We shall only point out how the computation is effectuated, since the formulas are quite long.
\begin{align*}
	\Reis(f_1,f_2) &:= (\Lambda_{\F}^*)^2 \cdot \left\{ \eis^{\reg,(2)}(\frac{1}{2},f_1f_2) + \frac{1}{2} \eis^{\reg,(1)}(\frac{1}{2}, f_1 \cdot \Intw_0^{(1)}f_2 + \Intw_0^{(1)}f_1 \cdot f_2) \right. \\
	&\quad\left. + \frac{1}{4} \eis^{\reg}(\frac{1}{2}, \Intw_0^{(1)}f_1 \cdot \Intw_0^{(1)}f_2) \right\}
\end{align*}
is the $\intL^2$-residue of $\eis^*(0,f_1) \cdot \eis^*(0,f_2)$. Let $\varphi := \eis^*(0,f_1) \cdot \eis^*(0,f_2) - \Reis(f_1,f_2)$, then we need to compute
	$$ \int_{[\PGL_2]}^{\reg} \varphi \cdot \eis^{\reg,(n)}(\frac{1}{2},f_3) + \int_{[\PGL_2]}^{\reg} \Reis(f_1,f_2) \cdot \eis^{\reg,(n)}(\frac{1}{2},f_3). $$
	The first term is computed by Proposition \ref{TripleToDouble} (2), involving
	$$ \int_{[\PGL_2]} \varphi = \int_{[\PGL_2]}^{\reg} \eis^*(0,f_1) \cdot \eis^*(0,f_2) = \frac{(\Lambda_{\F}^*)^2}{4} \int_{[\PGL_2]}^{\reg} \eis^{(1)}(0,f_1) \cdot \eis^{(1)}(0,f_2), $$
which is treated in Theorem \ref{RIPEisUnitary} (2). The second term is treated in Theorem \ref{RIPEisSing} (3).
\end{proof}

\section{Appendix: Bounds of Smooth Eisenstein Series}

	\subsection{General Remarks}
	
	We take the notations and assumptions in \cite{Wu5}. Namely we fix a section $s_{\F}: \ag{R}_+ \to \F^{\times} \backslash \ag{A}^{\times}$ and assume the Hecke characters $\omega, \xi$ to be trivial on the image of $s_{\F}$. We then have the definition of the Eisenstein series $\Eis(s,\xi,\omega\xi^{-1};f)$ for $f \in V_{\xi,\omega\xi^{-1}}^{\infty}$.
\begin{remark}
	We will sometimes omit $\xi,\omega\xi^{-1}$ and write $\Eis(s,f)$ when it is clear from the context.
\end{remark}
\noindent In \cite{Wu5}, we studied the size of $\Eis(s,\xi,\omega\xi^{-1};f)$. For the purpose of the present paper, we need something finer. Precisely, we shall decompose it as
	$$ \Eis(s,\xi,\omega\xi^{-1};f) = \eisCst(s,\xi,\omega\xi^{-1};f) + \left( \Eis(s,\xi,\omega\xi^{-1};f) - \eisCst(s,\xi,\omega\xi^{-1};f) \right) \text{ with} $$
	$$ \eisCst(s,\xi,\omega\xi^{-1};f)(g) := \int_{\F \backslash \ag{A}} \Eis(s,\xi,\omega\xi^{-1};f)(n(x)g) dx, n(x) = \begin{pmatrix} 1 & x \\ & 1 \end{pmatrix} $$
and study the growth in $g$ of $\eisCst(s,\xi,\omega\xi^{-1};f)$ and $\Eis(s,\xi,\omega\xi^{-1};f) - \eisCst(s,\xi,\omega\xi^{-1};f)$ separately, as well as all their derivatives with respect to $s$.

	The study of the constant term is reduced to the study of the intertwining operator, which is already done in \cite{Wu5}. We focus on
	$$ \Eis(s,\xi,\omega\xi^{-1};f)(g) - \eisCst(s,\xi,\omega\xi^{-1};f)(g) = \sum_{\alpha \in \F^{\times}} W(s,\xi,\omega\xi^{-1};f)(a(\alpha)g) \text{ with } a(\alpha)=\begin{pmatrix} \alpha & \\ & 1 \end{pmatrix} $$
	$$ \text{and} \quad W(s,\xi,\omega\xi^{-1};f)(g) = \int_{\F \backslash \ag{A}} \psi(-x) \Eis(s,\xi,\omega\xi^{-1};f)(n(x)g) dx, $$
where $\psi$ is the standard additive character of $\F \backslash \ag{A}$. We are thus reduced to the study of the Whittaker functions $W(s,\xi,\omega\xi^{-1};f)$. If we were only interested in $W(s,\xi,\omega\xi^{-1};f)$ itself, then its behavior is already completely clear by \cite{J04} or more generally with the ``singular'' cases by \cite[Proposition 2.2]{JS90}. However, we need a bit more for our purpose in this paper. Namely, we also need to estimate $\frac{\partial^n}{\partial s^n} W(s,\xi,\omega\xi^{-1};f)(g)$. Then not the results of \textit{loc.cit.} but the method serves, i.e., the method of integral representation of Whittaker functions.

	If $\Phi \in \Sch(\ag{A}^2)$ is a Schwartz function, we can define the following (non flat) section in $V_{s,\xi,\omega\xi^{-1}}^{\infty}$
	$$ f_{\Phi}(s,\xi,\omega\xi^{-1};g) = \xi(\det g)\norm[\det g]_{\ag{A}}^{\frac{1}{2}+s} \int_{\ag{A}^{\times}} \Phi((0,t)g) \omega^{-1}\xi^2(t) \norm[t]_{\ag{A}}^{1+2s} d^{\times}t $$
first defined for $\Re s > 0$ then meromorphically continued to $s \in \ag{C}$. Given $f \in V_{0,\xi,\omega\xi^{-1}}^{\infty}$, we want to give an explicit $\Phi$ associated with $f$. For simplicity of notations, we may assume $f$ to be a pure tensor. We then construct $\Phi = \otimes_v' \Phi_v$ place by place:

\noindent (1) At $\F_v = \ag{C}$ resp. $\F_v = \ag{R}$ and for $f_v$ spherical resp. not spherical, we choose $\Phi_v$ using the construction in \cite[Lemma 3.5 (1)]{Wu5} resp. \cite[Lemma 3.8 (1)]{Wu5} for spherical resp. smooth functions.

\noindent (2) At $v < \infty$ and for $f_v$ not spherical, we choose $\Phi_v$ by
	$$ \Phi_v((0,1) \kappa) = \Cond(\psi_v)\xi_v(\det \kappa)^{-1}f_v(\kappa), \kappa \in \gp{K}_v \text{ i.e.} $$
	$$ \Phi_v(ux,u) = \Cond(\psi_v)\omega_v \xi_v^2(u) f_v \begin{pmatrix} 1 & \\ x & 1 \end{pmatrix}, \forall u \in \vo_v^{\times}, x \in \vo_v; $$
	$$ \Phi_v(u,uy) = \Cond(\psi_v)\omega_v \xi_v^2(u) f_v \begin{pmatrix} & -1 \\ 1 & y \end{pmatrix}, \forall u \in \vo_v^{\times}, y \in \varpi_v \vo_v; $$
	and $\Phi_v(x,y)=0$ for $\max(\norm[x]_v,\norm[y]_v) \neq 1$.
	
\noindent (3) At $v < \infty$ and for $f_v$ spherical, we choose $\Phi_v$ by
	$$ \Phi_v = \Cond(\psi_v) \cdot f_v(1) \cdot 1_{\vo_v^2}. $$

Let $S=S(f)$ be the set of places $v$ such that $f_v$ is not spherical. Then we get
\begin{equation}
	f_{\Phi}(s,\xi,\omega\xi^{-1};g) = \Dis(\F)^{\frac{1}{2}}\Lambda^S(1+2s,\omega^{-1}\xi^2) \prod_{\substack{v \in S \\ v \mid \infty}} K_{v,a_v}(s,\omega_v^{-1}\xi_v^2) \cdot f_s(g).
\label{RelSec}
\end{equation}
\noindent We can thus deduce the bounds of $W(s,\xi,\omega\xi^{-1};f)$ from those of
\begin{equation}
	W_{\Phi}(s,\xi,\omega\xi^{-1}; g) = \xi(\det g)\norm[\det g]_{\ag{A}}^{\frac{1}{2}+s} \int_{\ag{A}^{\times}} \Four[2]{\rpR(g).\Phi}(t,\frac{1}{t}) \omega^{-1}\xi^2(t)\norm[t]_{\ag{A}}^{2s} d^{\times}t,
\label{SchWhi}
\end{equation}
where the partial Fourier transforms are defined as in \cite[(3.3)]{Wu5}.

	In Section 4.2, we will bound (\ref{SchWhi}) locally place by place. We then use the obtained bound to get a bound for the sum $\Sigma_{\alpha \in \F^{\times}} \norm[W_{\Phi}(s,\xi,\omega\xi^{-1}; a(\alpha)g)]$, using a convergence lemma treated in Section 4.4. We will treat all bounds with uniformity for $s$ with real part lying in any compact interval, so that the bounds for the derivatives in $s$ follow automatically by Cauchy's integral formulae.
	
	In Section 4.3, we will determine the behavior of the constant term based on \cite{Wu5}.

	\subsection{Bounds of Non Constant Terms}
	
		\subsubsection{Archimedean Places}
		
	We omit the subscript $v$ since we work locally. The local integral representation has the form
	$$ W_{\Phi}(s,\xi,\omega\xi^{-1}; a(y)\kappa) = \omega\xi^{-1}(y)\norm[y]_{\F}^{\frac{1}{2}-s} \int_{\F^{\times}} \Four[2]{\rpR(\kappa).\Phi}(t,\frac{y}{t}) \omega^{-1}\xi^2(t)\norm[t]_{\F}^{2s} d^{\times}t, \text{ for } y \in \F^{\times}, \kappa \in \gp{K}. $$
We are thus reduced to studying the integral at the right hand side. By \cite[Proposition 4.1]{J04} as well as its counterpart in the singular cases, it is easy to see the rapid decay at $\infty$ of
	$$ \norm[W_{\Phi}(s,\xi,\omega\xi^{-1}; a(y)\kappa)] \ll \norm[y]_{\F}^{-N}, \forall N \in \ag{N}, $$
and the polynomial increase at $0$ of
	$$ \norm[W_{\Phi}(s,\xi,\omega\xi^{-1}; a(y)\kappa)] \ll_{\epsilon} \norm[y]_{\F}^{\frac{1}{2} - \norm[\Re s] - \epsilon}, \forall \epsilon > 0. $$
As for the implied constants in the above estimations, one naturally guess it is related to the Schwartz norms of $\Four[2]{\rpR(\kappa).\Phi}$. Then we need to related these norms to the Schwartz norms of $\Phi$ itself. According to this strategy, we state the following two lemmas and the desired proposition.
\begin{lemma}
	For any Schwartz norm $\Sch^*$ there is a Schwartz norm $\Sch^{**}$ such that
	$$ \sup_{\kappa \in \gp{K}} \Sch^*(\Four[2]{\rpR(\kappa).\Phi}) \ll \Sch^{**}(\Phi). $$
\label{SchNormEquivA}
\end{lemma}
\begin{lemma}
	For the real part of $s$ lying in a fixed compact interval, any Schwartz function $\Phi \in \Sch(\F^2)$ and any integer $N \in \ag{N}$, there is a Schwartz norm $\Sch^*$ such that as $\norm[y]_{\F} \to \infty$
	$$ \extnorm{  \int_{\F^{\times}} \Phi(t,\frac{y}{t}) \omega^{-1}\xi^2(t)\norm[t]_{\F}^{2s} d^{\times}t } \ll \Sch^*(\Phi) \norm[y]_{\F}^{-N}; $$
while for any $\epsilon > 0$ there is a Schwartz norm $\Sch^{**}$ such that as $\norm[y]_{\F} \to 0$
	$$ \extnorm{  \int_{\F^{\times}} \Phi(t,\frac{y}{t}) \omega^{-1}\xi^2(t)\norm[t]_{\F}^{2s} d^{\times}t } \ll_{\epsilon} \Sch^{**}(\Phi) \max( \norm[y]_{\F}^{-\epsilon}, \norm[y]_{\F}^{-2 \Re s - \epsilon}). $$
\label{IntBdA}
\end{lemma}
\begin{proposition}
	Let the real part of $s$ vary in a fixed compact interval. For any integer $N \in \ag{N}$, as $\norm[y] \to \infty$ and uniformly in $\kappa$, there is a Schwartz norm $\Sch^*$ such that
	$$ \norm[W_{\Phi}(s,\xi,\omega\xi^{-1}; a(y)\kappa)] \ll \Sch^*(\Phi) \norm[y]_{\F}^{-N}; $$
while for any $\epsilon > 0$, as $\norm[y] \to 0$ and uniformly in $\kappa$, there is a Schwartz norm $\Sch^{**}$ such that
	$$ \norm[W_{\Phi}(s,\xi,\omega\xi^{-1}; a(y)\kappa)] \ll_{\epsilon} \Sch^{**}(\Phi) \norm[y]_{\F}^{\frac{1}{2} - \norm[\Re s] - \epsilon}. $$
\label{LocWhiBdA}
\end{proposition}
\noindent We recall the definition of Schwartz norms on $\ag{R}^d$ for positive integers $d$.
\begin{definition}
	For $l \in [1,\infty]$, $\vec{p}, \vec{m} \in \ag{N}^d$, we put the semi-norm $\Sch_l^{\vec{p},\vec{m}}$ on $\Sch(\ag{R}^d)$ by
	$$ \Sch_l^{\vec{p},\vec{m}}(\Phi) = \extNorm{ \vec{x}^{\vec{p}} \cdot \partial^{\vec{m}} \Phi }_l. $$
Here we have written:
\begin{itemize}
	\item $\Norm_l$ is the $\intL^l$-norm on $\ag{R}^d$.
	\item For $\vec{x}=(x_i)_{1 \leq i \leq d} \in \ag{R}^d, \vec{p}=(p_i)_{1 \leq i \leq d} \in \ag{N}^d$, $\vec{x}^{\vec{p}}=\Pi_{i=1}^d x_i^{p_i}$.
	\item For $\vec{n}=(n_i)_{1 \leq i \leq d} \in \ag{N}^d, \vec{x}=(x_i)_{1 \leq i \leq d} \in \ag{R}^d$,
	$$ \partial^{\vec{n}} = \prod_{i=1}^d \frac{\partial^{n_i}}{\partial x_i^{n_i}}. $$
\end{itemize}
\end{definition}
\begin{remark}
	Since $\ag{C} \simeq \ag{R}^2$, we put the semi-norms for $\Sch(\ag{R}^2)$ on $\Sch(\ag{C})$.
\end{remark}
\begin{remark}
	If we do not specify the parameters of a Schwartz norm $\Sch^*$, we mean the max of a finite collection of Schwartz norms. This applies to Lemma \ref{SchNormEquivA}, \ref{IntBdA} and Proposition \ref{LocWhiBdA}.
\end{remark}

	We first treat Lemma \ref{SchNormEquivA}.
\begin{proposition}
	The topology on $\Sch(\ag{R}^d)$ defined by the system of semi-norms $\Sch_l^*$ does not depend on $l \in [1,\infty]$.
\end{proposition}
\begin{proof}
	In the case $d=1$, we have for $l \in [1,\infty)$ and any $\Phi \in \Sch(\ag{R})$
	$$ \int_{\ag{R}} \norm[\Phi(x)]^l dx \leq \Sch_{\infty}^{0,0}(\Phi)^l \int_{-1}^1 dx + \Sch_{\infty}^{2,0}(\Phi)^l \int_{\norm[x]>1} \norm[x]^{-2l} dx, $$
from which we deduce by replacing $\Phi(x)$ with $x^p \Phi^{(m)}(x)$ that
	$$ \Sch_l^{p,m}(\Phi) \ll_l \Sch_{\infty}^{p,m}(\Phi) + \Sch_{\infty}^{p+2,m}(\Phi). $$
	In the opposite direction, from H\"older inequality
	$$ \norm[\Phi(y) - \Phi(x)] = \extnorm{ \int_x^y \Phi'(t) dt} \leq \Norm[\Phi']_l \cdot \left( \int_x^y dt \right)^{\frac{1-l}{l}} $$
and $(a+b)^l \leq 2^{l-1}(a^l + b^l)$ we deduce
	$$ \norm[\Phi(y)]^l \leq 2^{l-1} \left( \norm[\Phi(x)]^l + \Norm[\Phi']_l^l \cdot \norm[y-x]^{l-1} \right). $$
	Integrating both sides against $\min(1,\norm[x-y]^{-l-1})dx \leq dx$ gives
	$$ (2+\frac{2}{l}) \norm[\Phi(y)]^l \leq 2^{l-1} \left( \Norm[\Phi]_l^l + \Norm[\Phi']_l^l \cdot \int_{\ag{R}} \min(\norm[x]^{l-1}, \norm[x]^{-2}) dx \right). $$
	Hence we get (a Sobolev inequality) and conclude the case $d=1$ by
	$$ \Norm[\Phi]_{\infty} \ll_l \Norm[\Phi]_l + \Norm[\Phi']_l. $$
	
	For general $d$, one deduces easily by induction
	$$ \Sch_l^{\vec{p},\vec{m}}(\Phi)^l \ll_{l,d} \sum_{\vec{\epsilon} \in \{ 0,2 \}^d} \Sch_{\infty}^{\vec{p}+\vec{\epsilon},\vec{m}}(\Phi)^l, $$
	$$ \Norm[\Phi]_{\infty}^l \ll_{l,d} \sum_{\vec{\epsilon} \in \{ 0,1 \}^d} \Sch_l^{0,\vec{\epsilon}}(\Phi)^l. $$
\end{proof}
\begin{proof}{(of Lemma \ref{SchNormEquivA})}
	By the above proposition, the problem is reduced to the uniform continuity of
	$$ \Four[2]{\cdot} \circ \rpR(\kappa): \Sch(\F^2) \to \Sch(\F^2) $$
with respect to $\kappa \in \gp{K}$. The continuity of $\Four[2]{\cdot}$ follows by considering the $\Sch_2^*$ semi-norms. The uniform continuity of $\rpR(\kappa)$ follows by considering the $\Sch_{\infty}^*$ semi-norms.
\end{proof}

	We then turn to Lemma \ref{IntBdA}. Actually, we are going to reduce to the situation of Mellin transform on $\ag{R}_+$, which we shall study at the first place. For any $c \in \ag{R}$, define
	$$ \fsB_c(\ag{R}_+) = \left\{ f: \in \Cont^{\infty}(\ag{R}_+): \sup_{y > 0} \norm[f^{(k)}(y) y^{\sigma+k}] < \infty, \forall \sigma > c \right\}. $$
	$$ \fsH_c(\ag{C}) = \left\{ M \text{ holomorphic in } \Re s > c: \sup_{\Re s = \sigma} \norm[s(s+1)\cdots (s+k-1)M(s)] < \infty, \forall \sigma > c \right\}. $$
\begin{definition}
	For any fixed $l \in [0,\infty]$, we put a system of semi-norms $B_l^{k,\sigma}$ with $k \in \ag{N}, \sigma \in (c,\infty)$ on $\fsB_c(\ag{R}_+)$ by
	$$ B_l^{k,\sigma}(f) = \left( \int_0^{\infty} \norm[f^{(k)}(y) y^{\sigma+k}]^l \frac{dy}{y} \right)^{\frac{1}{l}}, l \neq \infty; B_{\infty}^{k,\sigma}(f) = \sup_{y > 0} \norm[f^{(k)}(y) y^{\sigma+k}]. $$
\end{definition}
\begin{definition}
	For any fixed $l \in [0,\infty]$, we put a system of semi-norms $H_l^{k,\sigma}$ with $k \in \ag{N}, \sigma \in (c,\infty)$ on $\fsH_c(\ag{C})$ by
	$$ H_l^{k,\sigma}(M) = \left( \int_{\Re s = \sigma} \norm[s(s+1)\cdots (s+k-1)M(s)]^l \frac{ds}{2\pi i} \right)^{\frac{1}{l}}, l \neq \infty; $$
	$$ H_{\infty}^{k,\sigma}(M) = \sup_{\Re s = \sigma} \norm[s(s+1)\cdots (s+k-1)M(s)]. $$
\end{definition}
\begin{proposition}
	The topology on $\fsB_c(\ag{R}_+)$ defined by $B_l^{k,\sigma}$ does not depend on $l$. More precisely, for any $f \in \fsB_c(\ag{R}_+)$ we have for $1 \leq l < \infty$ and $\epsilon > 0$ small with $\sigma - \epsilon > c$
	$$ B_l^{k,\sigma}(f) \ll_{\epsilon, l} B_{\infty}^{k,\sigma + \epsilon}(f) + B_{\infty}^{k,\sigma - \epsilon}(f); $$
	$$ B_{\infty}^{k,\sigma}(f) \ll_{\epsilon, l, \sigma + k} B_l^{k,\sigma }(f) + B_l^{k+1,\sigma}(f). $$
\end{proposition}
\begin{proof}
	The first inequality follows from
	$$ \int_0^{\infty} \norm[f^{(k)}(y) y^{\sigma+k}]^l \frac{dy}{y} \leq B_{\infty}^{k,\sigma + \epsilon}(f)^l \int_1^{\infty} y^{-\epsilon l} \frac{dy}{y} + B_{\infty}^{k,\sigma - \epsilon}(f)^l \int_0^1 y^{\epsilon l} \frac{dy}{y}. $$
	For the second inequality, we first note that for any $x,y > 0$
	$$ f^{(k)}(y)y^{\sigma+k} - f^{(k)}(x)x^{\sigma+k} = \int_y^x f^{(k+1)}(u) u^{\sigma+k+1} \frac{du}{u} + (\sigma + k) \int_y^x f^{(k)}(u) u^{\sigma+k} \frac{du}{u}. $$
	We can bound the integrals using H\"older inequality as
	$$ \extnorm{\int_y^x f^{(k)}(u) u^{\sigma+k} \frac{du}{u}} \leq B_l^{k,\sigma}(f) \cdot \norm[\log (y/x)]^{\frac{l-1}{l}}, $$
from which we deduce
	$$ \norm[f^{(k)}(y)y^{\sigma+k}] \leq \norm[f^{(k)}(x)x^{\sigma+k}] + \left[ B_l^{k+1,\sigma}(f) + \norm[\sigma+k] \cdot B_l^{k,\sigma}(f) \right] \cdot \norm[\log (y/x)]^{\frac{l-1}{l}}. $$
	Raising to the power $l \geq 1$ and use $(a+b)^l \leq 2^{l-1}(a^l + b^l)$ gives
	$$ \norm[f^{(k)}(y)y^{\sigma+k}]^l \leq 2^{l-1} \left\{ \norm[f^{(k)}(x)x^{\sigma+k}]^l + \left[ B_l^{k+1,\sigma}(f) + \norm[\sigma+k] \cdot B_l^{k,\sigma}(f) \right]^l \cdot \norm[\log (y/x)]^{l-1} \right\}. $$
	Integrating both sides against $\min((x/y)^{\epsilon l}, (x/y)^{-\epsilon l}) dx/x \leq dx/x$ gives
	$$ \frac{2}{\epsilon l} \cdot \norm[f^{(k)}(y)y^{\sigma+k}]^l \leq 2^{l-1} \left\{ B_l^{k,\sigma}(f)^l + \left[ B_l^{k+1,\sigma}(f) + \norm[\sigma+k] \cdot B_l^{k,\sigma}(f) \right]^l \cdot \int_0^{\infty} \min(x^{\epsilon},x^{-\epsilon}) \norm[\log x]^{l-1} \frac{dx}{x} \right\}. $$
	We conclude since $\int_0^{\infty} \min(x^{\epsilon},x^{-\epsilon}) \norm[\log x]^{l-1} \frac{dx}{x} < \infty$.
\end{proof}
\begin{proposition}
	The topology on $\fsH_c(\ag{C})$ defined by $H_l^{k,\sigma}$ does not depend on $l$. More precisely, for any $M \in \fsH_c(\ag{C})$ we have for $1 \leq l < \infty$ and $\epsilon > 0$ small with $\sigma - \epsilon > c$
	$$ H_l^{k,\sigma}(M) \ll_{k, l} H_{\infty}^{k,\sigma}(M) + H_{\infty}^{k+2,\sigma}(M); $$
	$$ H_{\infty}^{k,\sigma}(M) \ll_{\epsilon, l, \sigma + k} H_l^{k,\sigma+\epsilon}(M) + H_l^{k+1,\sigma+\epsilon}(M) + H_l^{k,\sigma-\epsilon}(M) + H_l^{k+1,\sigma-\epsilon}(M). $$
\end{proposition}
\begin{proof}
	The first inequality follows from
\begin{align*}
	\int_{\Re s = \sigma} \extnorm{s(s+1)\cdots (s+k-1)M(s)}^l \frac{ds}{2\pi i} &\leq H_{\infty}^{k,\sigma}(M)^l \int_{\substack{\Re s = \sigma \\ \norm[\Im s] \leq 1}} \frac{ds}{2\pi i} \\
	&\quad + H_{\infty}^{k+2,\sigma}(M)^l \int_{\substack{\Re s = \sigma \\ \norm[\Im s] > 1}} \frac{1}{\norm[(s+k)(s+k+1)]^l} \frac{ds}{2\pi i}.
\end{align*}
	For the second inequality, we first note that for any $s_0$ with $\Re s_0 = \sigma$
\begin{align*}
	s_0(s_0+1)\cdots (s_0+k-1)M(s_0) &= \int_{\Re s = \sigma + \epsilon} \frac{s(s+1)\cdots (s+k-1)M(s)}{s-s_0} \frac{ds}{2\pi i} \\
	&\quad - \int_{\Re s = \sigma - \epsilon} \frac{s(s+1)\cdots (s+k-1)M(s)}{s-s_0} \frac{ds}{2\pi i}.
\end{align*}
	To bound the integrals, we apply H\"older inequality to get
\begin{align*}
	\int_{\Re s = \sigma + \epsilon} \extnorm{\frac{s(s+1)\cdots (s+k-1)M(s)}{s-s_0}} \frac{ds}{2\pi i} &\leq \frac{H_l^{k+1,\sigma+\epsilon}(M)}{\epsilon} \cdot \left( \int_{{\substack{\Re s = \sigma+\epsilon \\ \norm[\Im s] \geq 1}}} \frac{1}{\norm[s+k]^{\frac{l}{l-1}}} \frac{ds}{2\pi i} \right)^{\frac{l-1}{l}} \\
	&\quad + \frac{H_l^{k,\sigma+\epsilon}(M)}{\epsilon} \cdot \left( \int_{{\substack{\Re s = \sigma+\epsilon \\ \norm[\Im s] \leq 1}}} \frac{ds}{2\pi i} \right)^{\frac{l-1}{l}},
\end{align*}
	and conclude by the similar bound on $\Re s = \sigma - \epsilon$.
\end{proof}
\begin{proposition}
	The two maps
	$$ \fsB_c(\ag{R}_+) \to \fsH_c(\ag{C}), f \mapsto \Mellin{f}(s):= \int_0^{\infty} f(y) y^s \frac{dy}{y}, \text{ for } \Re s > c; $$
	$$ \fsH_c(\ag{C}) \to \fsB_c(\ag{R}_+), M \mapsto f_M(y) := \int_{\Re s = \sigma} M(s)y^{-s} \frac{ds}{2\pi i}, \forall \sigma > c $$
	are continuous with respect to the above topologies defined by semi-norms.
\end{proposition}
\begin{proof}
	By integration by parts we get
	$$ \Mellin{f}(s) = \frac{(-1)^k}{s(s+1) \cdots (s+k-1)} \int_0^{\infty} f^{(k)}(y) y^{s+k} \frac{dy}{y}, $$
from which it follows readily that $H_{\infty}^{k,\sigma}(\Mellin{f}) \leq B_1^{k,\sigma}(f)$. Similarly we pass the derivatives under the integral to get
	$$ f_M^{(k)}(y) = (-1)^k \int_{\Re s = \sigma} s(s+1) \cdots (s+k-1) M(s)y^{-s-k} \frac{ds}{2\pi i}, $$
from which it follows readily that $B_{\infty}^{k,\sigma}(f_M) \leq H_1^{k,\sigma}(M)$.
\end{proof}
\begin{definition}
	We write the multiplicative group $\F^1 = \{ x \in \F: \norm[x]_{\F} = 1 \}$. For any function $f$ on $\F^{\times}$ and any character $\xi \in \widehat{\F^1}$ we define a function on $\ag{R}_+$
	$$ f_{\xi}(t) = f(t;\xi) = \int_{\F^1} f(tu) \xi(u) du, t > 0 $$
where $du$ is the probability Haar measure on $\F^1$. Concretely:
\begin{itemize}
	\item[(1)] If $\F=\ag{R}$ then $\F^1=\{ \pm 1\}, \widehat{\F^1}=\{ \xi_+, \xi_- \}$ with $\xi_+ \equiv 1$ and $\xi_-(-1)=-1$. We then define
	$$ f(t;+) = f_+(t) = \frac{1}{2} \left( f(t) + f(-t) \right), f(t;-) = f_-(t) = \frac{1}{2} \left( f(t) - f(-t) \right). $$
	\item[(2)] If $\F=\ag{C}$ then $\F^1=\{ e^{i \theta}: \theta \in \ag{R}/2\pi \ag{Z} \}, \widehat{\F^1}=\{ \xi_n: n \in \ag{Z} \}$ with $\xi_n(e^{i\theta}) = e^{i n \theta}$. We then define
	$$ f(t;n) = f_n(t) = \int_{\ag{R} / 2\pi \ag{Z}} f(te^{i\theta}) e^{in\theta} \frac{d\theta}{2\pi}. $$
\end{itemize}
\label{FourF^1}
\end{definition}
\begin{proposition}
	$f_{\xi} \in \fsB_0(\ag{R}_+)$ for any $f \in \Sch(\F)$ and $\xi \in \widehat{\F^1}$. The map
	$$ \Sch(\F) \to \fsB_0(\ag{R}_+), f \mapsto f_{\xi} $$
is continuous. Moreover, in the case $\F=\ag{C}$, for any $k,l \in \ag{N}, \sigma > 0$ there is a finite collection of norms $\Sch_{\infty}^*$ independent of $n$ ($n \neq 0$ if $l \neq 0$) such that
	$$ B_{\infty}^{k,\sigma}(f_n) \ll_{k,\sigma} \norm[n]^{-l} \Sch_{\infty}^*(f). $$
\end{proposition}
\begin{proof}
	In the case $\F=\ag{R}$, we have
	$$ t^k\frac{d^k}{dt^k} f(t;+) = \frac{1}{2} \left( t^kf^{(k)}(t) + (-1)^k t^k f^{(k)}(-t) \right), $$
	$$ t^k\frac{d^k}{dt^k} f(t;-) = \frac{1}{2} \left( t^kf^{(k)}(t) + (-1)^{k+1} t^k f^{(k)}(-t) \right), $$
from which it is easy to see
	$$ B_{\infty}^{k,\sigma}(f_{\pm}) \leq \Sch_{\infty}^{\lfloor k+\sigma \rfloor, k}(f) + \Sch_{\infty}^{\lceil k+\sigma \rceil, k}(f), \forall k \in \ag{N}, \sigma > 0. $$
	In the case $\F=\ag{C}$, with the Cartesian \& Polar coordinates
	$$ (x,y) = (t \cos \theta, t \sin \theta), (z,\bar{z})=(x+iy,x-iy) $$
	we have
	$$ \frac{\partial}{\partial \theta} = i \left( z \frac{\partial}{\partial z} - \bar{z} \frac{\partial}{\partial \bar{z}} \right); t\frac{\partial}{\partial t} = z \frac{\partial}{\partial z} + \bar{z} \frac{\partial}{\partial \bar{z}}. $$
	By induction on $k \in \ag{N}$, it is easy to see
	$$ t^k \frac{\partial^k}{\partial t^k} = P_k(z \frac{\partial}{\partial z} + \bar{z} \frac{\partial}{\partial \bar{z}}) $$
for some polynomial $P_k \in \ag{Z}[X]$ and any $k \in \ag{N}$. It follows that
\begin{align*}
	t^k f_n^{(k)}(t) &= \int_{\ag{R} / 2\pi \ag{Z}} (P_k(z \frac{\partial}{\partial z} + \bar{z} \frac{\partial}{\partial \bar{z}})f)(te^{i\theta}) e^{in\theta} \frac{d\theta}{2\pi} \\
	&= \frac{(-1)^l}{n^l} \int_{\ag{R} / 2\pi \ag{Z}} (\left( z \frac{\partial}{\partial z} - \bar{z} \frac{\partial}{\partial \bar{z}} \right)^l P_k(z \frac{\partial}{\partial z} + \bar{z} \frac{\partial}{\partial \bar{z}})f)(te^{i\theta}) e^{in\theta} \frac{d\theta}{2\pi}.
\end{align*}
	Hence, we deduce that
\begin{align*}
	B_{\infty}^{k,\sigma}(f_n) &\leq \norm[n]^{-l} \left\{ \extNorm{(z\bar{z})^{\lfloor \frac{\sigma}{2} \rfloor}\left( z \frac{\partial}{\partial z} - \bar{z} \frac{\partial}{\partial \bar{z}} \right)^l P_k(z \frac{\partial}{\partial z} + \bar{z} \frac{\partial}{\partial \bar{z}})f}_{\infty} \right. \\
	&\quad \left. + \extNorm{(z\bar{z})^{\lceil \frac{\sigma}{2} \rceil}\left( z \frac{\partial}{\partial z} - \bar{z} \frac{\partial}{\partial \bar{z}} \right)^l P_k(z \frac{\partial}{\partial z} + \bar{z} \frac{\partial}{\partial \bar{z}})f}_{\infty} \right\}.
\end{align*}
	The right hand side is obviously bounded by some Schwartz norm of $f$.
\end{proof}
\begin{proof}{(of Lemma \ref{IntBdA})}
	We only treat the case $\F=\ag{C}$, the real case being similar and simpler. Writing
	$$ f(y) = \int_{\ag{C}^{\times}} \Phi(t,\frac{y}{t}) \omega^{-1}\xi^2(t)\norm[t]_{\ag{C}}^{2s} d^{\times}t, $$
we can take its Fourier expansion on $\ag{C}^1$
	$$ f(te^{i\theta}) = \sum_{n \in \ag{Z}} f_n(t) e^{-in\theta}, t \in \ag{R}_+. $$
	Extending each $\xi_n \in \widehat{\ag{C}^1}$ to $\ag{C}^{\times}$ by triviality on $\ag{R}_+$ we have the Mellin transform
\begin{align*}
	\Mellin{f_n}(s_1) &= \int_{\ag{C}^{\times}} f(y) \xi_n(y) \norm[y]_{\ag{C}}^{\frac{s_1}{2}} d^{\times}y \\
	&= \int_{\ag{C}^{\times} \times \ag{C}^{\times}} \Phi(t,y) \omega^{-1}\xi^2\xi_n(t) \norm[t]_{\ag{C}}^{\frac{s_1}{2}+2s} \xi_n(y) \norm[y]_{\ag{C}}^{\frac{s_1}{2}} d^{\times}t d^{\times}y \\
	&= \int_{\ag{R}_+ \times \ag{R}_+} \Phi_{\omega^{-1}\xi^2\xi_n, \xi_n}(t_1,t_2) \norm[t_1]_{\ag{C}}^{\frac{s_1}{2}+2s+i\mu(\omega^{-1}\xi^2)} \norm[t_2]_{\ag{C}}^{\frac{s_1}{2}} d^{\times}t_1 d^{\times}t_2.
\end{align*}
	Considering the $H_{\infty}^*$ semi-norms it is easy to see $\Mellin{f_n} \in \fsH_{\max(0,-4\Re s)}(\ag{C})$. We can also bound
	$$ H_{\infty}^{k,\sigma}(\Mellin{f_n}) \ll_{k,\sigma} \min(1, \norm[n]^{-2}) \Sch_1^*(\Phi), \forall \sigma > \max(0,-4\Re s). $$
	As $\Re s$ lies in a compact interval, the orders of $\Sch_1^*$ can be made uniform (but depends on $\sigma$). Hence $f_n \in \fsB_{\max(0,-4\Re s)}(\ag{R}_+)$ and for any $\sigma > \max(0,-4\Re s)$ we get
	$$ \norm[t^{\sigma}f(te^{i\theta})] \leq \sum_n B_{\infty}^{0,\sigma}(f_n) \ll \Sch_1^*(\Phi) \sum_n \min(1, \norm[n]^{-2}). $$
	We conclude by noting $t^{\sigma} = \norm[te^{i\theta}]_{\ag{C}}^{\sigma/2}$.
\end{proof}

	Obviously, Proposition \ref{LocWhiBdA} is a direct consequence of Lemma \ref{SchNormEquivA} and \ref{IntBdA}.

		\subsubsection{Non Archimedean Places}
		
	We continue to omit the subscript $v$ for simplicity of notations.
\begin{definition}
	Let $d \geq 1$ be an integer. For any $\Phi \in \Sch(\F^d)$ we define its \emph{support index} $D(\Phi) \in \ag{Z}$, \emph{additive invariance index} $\delta(\Phi) \in \ag{Z}$ and \emph{multiplicative invariance index} $m(\Phi) \in \ag{N}$ as follows.
\begin{itemize}
	\item[(1)] $D(\Phi)$ is the largest integer $D$ such that
	$$ \Phi(\vec{x}) \neq 0 \Rightarrow \vec{x} \in \vp^D \times \cdots \times \vp^D. $$
	\item[(2)] $\delta(\Phi)$ is the smallest integer $\delta$ such that
	$$ \Phi(\vec{x}+\vec{t}) = \Phi(\vec{x}), \forall \vec{x} \in \F^d, \vec{t} \in \vp^{\delta} \times \cdots \times \vp^{\delta}. $$
	\item[(3)] $m(\Phi)$ is the smallest integer $m \geq 0$ such that for any $\kappa \in \GL_d(\vo)$ with $\kappa-1 \in \Mat_d(\vp^m)$
	$$ \rpR(\kappa).\Phi(\vec{x}) = \Phi(\vec{x}.\kappa) = \Phi(\vec{x}), \forall \vec{x} \in \F^d. $$
\end{itemize}
\end{definition}
\begin{proposition}
	The three indices satisfy the following relations.
\begin{itemize}
	\item[(0)] $m(\Phi) \leq \delta(\Phi)-D(\Phi)$.
	\item[(1)] For any $\kappa \in \GL_d(\vo)$, we have $D(\rpR(\kappa).\Phi)=D(\Phi)$, $\delta(\rpR(\kappa).\Phi)=\delta(\Phi)$ and $m(\rpR(\kappa).\Phi)=m(\Phi)$.
	\item[(2)] Let $\Four{\cdot}$ denote the Fourier transform
	$$ \Four{\Phi}(\vec{x}) = \int_{\F^d} \Phi(\vec{y}) \psi(-\vec{y}.\vec{x}) d\vec{y}. $$
	Then we have
	$$ D(\Phi)+\delta(\Four{\Phi}) = \delta(\Phi) + D(\Four{\Phi}) = -\cond(\psi). $$
	\item[(3)] More generally, let $I = \{ i_1, \cdots, i_j \} \subset \{ 1, \dots, d \}$. We define the partial Fourier transform $\Four[I]{\cdot} = \Four[i_1]{\Four[i_2]{ \cdots \Four[i_j]{\cdot}}}$. Then we have
	$$ \delta(\Four[I]{\Phi})) \leq \max(\delta(\Phi), -\cond(\psi)-D(\Phi)); $$
	$$ D(\Four[I]{\Phi}) \geq \min(D(\Phi), -\cond(\psi)-\delta(\Phi)). $$
\end{itemize}
\label{DdmRel}
\end{proposition}
\begin{proof}
	(0) and (1) are obvious from definition. (3) follows easily from (2). We prove (2) as follows. From
	$$ \Four{\Phi}(\vec{x}+\vec{t}) = \int_{\vp^{D(\Phi)} \times \cdots \times \vp^{D(\Phi)}} \Phi(\vec{y}) \psi(-\vec{y}.\vec{x}) \psi(-\vec{y}.\vec{t}) d\vec{y}, $$
we see that for $\vec{t} \in \vp^{-\cond(\psi)-D(\Phi)}$, $\vec{y}.\vec{t} \in \vp^{-\cond(\psi)}$ hence $\psi(-\vec{y}.\vec{t})=1$ and $\Four{\Phi}(\vec{x}+\vec{t}) = \Four{\Phi}(\vec{x})$. Thus
	$$ \delta(\Four{\Phi}) \leq -\cond(\psi)-D(\Phi). $$
	On the other hand, if $\vec{x} \notin \vp^{-\cond(\psi)-\delta(\Phi)} \times \cdots \times \vp^{-\cond(\psi)-\delta(\Phi)}$ then at least for one component, say $x_1 \notin \vp^{-\cond(\psi)-\delta(\Phi)}$, i.e., $v(x_1) < -\cond(\psi)-\delta(\Phi)$. As $t_1$ runs under the condition $v(t_1)=-\cond(\psi)-1-v(x_1) \geq \delta(\Phi)$, $x_1t_1$ runs under the condition $v(x_1t_1)=-\cond(\psi)-1$. Hence at least for one $t_1$, $\psi(x_1t_1) \neq 1$. Writing $\vec{t}=(t_1,0,\dots,0)$, we get
	$$ \Four{\Phi}(\vec{x}) = \int_{\F^d} \Phi(\vec{y}+\vec{t}) \psi(-\vec{y}.\vec{x}) d\vec{y} = \int_{\F^d} \Phi(\vec{y}) \psi(-\vec{y}.\vec{x}) d\vec{y} \cdot \psi(\vec{t}.\vec{x}) = \psi(t_1x_1) \Four{\Phi}(\vec{x}). $$
	Hence $\Four{\Phi}(\vec{x}) = 0$ and
	$$ D(\Four{\Phi}) \geq -\cond(\psi)-\delta(\Phi). $$
	Replacing $\Phi$ with $\Four{\Phi}$ in the above argument gives the inequalities in the opposite direction.
\end{proof}
\begin{definition}
	For $l \in [1,\infty]$, $\vec{\sigma} \in \ag{R}_{\geq 0}^d$, we put the semi-norm $\Sch_l^{\vec{\sigma}}$ on $\Sch(\F^d)$ by
	$$ \Sch_l^{\vec{\sigma}}(\Phi) = \extNorm{ \norm[\vec{x}^{\vec{\sigma}}]_{\F} \cdot \Phi }_l. $$
Here we have written:
\begin{itemize}
	\item $\Norm_l$ is the $\intL^l$-norm on $\F^d$.
	\item For $\vec{x}=(x_i)_{1 \leq i \leq d} \in \F^d, \vec{\sigma}=(\sigma_i)_{1 \leq i \leq d} \in \ag{R}_{\geq 0}^d$, $\norm[\vec{x}^{\vec{\sigma}}]_{\F}=\Pi_{i=1}^d \norm[x_i]_{\F}^{\sigma_i}$.
\end{itemize}
	We shall also write $\norm[\vec{\sigma}] = \sum_i \sigma_i$.
\end{definition}
\begin{proposition}
	We have the following relations of norms for any $\Phi \in \Sch(\F^d)$.
\begin{itemize}
	\item[(1)] $\Norm[\Phi]_{\infty} \leq q^{\frac{d\delta(\Phi)}{l}} \Norm[\Phi]_l$ and $\Norm[\Phi]_l \leq q^{-\frac{dD(\Phi)}{l}} \Norm[\Phi]_{\infty}$.
	\item[(2)] $\Sch_l^{\vec{\sigma}}(\Phi) \leq q^{-\norm[\vec{\sigma}] D(\Phi)} \Norm[\Phi]_l$.
\end{itemize}
\label{SchEquiv}
\end{proposition}
\begin{proof}
	It suffices to prove the case $d=1$. Let $x_0 \in \F$ be such that $\norm[\Phi(x_0)] = \Norm[\Phi]_{\infty}$, then
	$$ \Norm[\Phi]_{\infty}^l = \Vol(\vp^{\delta(\Phi)})^{-1} \int_{x_0 + \vp^{\delta(\Phi)}} \norm[\Phi(x)]^l dx \leq q^{\delta(\Phi)} \Norm[\Phi]_l^l $$
and we get the first inequality. The second follows from
	$$ \Norm[\Phi]_l^l \leq \Norm[\Phi]_{\infty}^l \int_{{\rm supp}(\Phi)} dx \leq \Norm[\Phi]_{\infty}^l \cdot \Vol(\vp^{D(\Phi)}) = q^{-D(\Phi)} \Norm[\Phi]_{\infty}^l. $$
	For the last, we deduce it from
	$$ \Sch_l^{\sigma}(\Phi)^l = \int_{\vp^{D(\Phi)}} \norm[x]_{\F}^{\sigma l} \norm[\Phi(x)]^l dx \leq \sup_{x \in \vp^{D(\Phi)}} \norm[x]_{\F}^{\sigma l} \cdot \Norm[\Phi]_l^l = q^{-\sigma l D(\Phi)} \Norm[\Phi]_l^l. $$
\end{proof}
\begin{definition}
	For any $c \in \ag{R}$, we define $\fsB_c(\ag{Z};\varpi)$ to be the space of functions $f: \varpi^{\ag{Z}} \to \ag{C}$ satisfying
\begin{itemize}
	\item[(1)] $\lim_{n \to -\infty} f(\varpi^n) q^{-n\sigma} = 0$ for any $\sigma > 0$.
	\item[(2)] $\lim_{n \to +\infty} f(\varpi^n) q^{-n \sigma} = 0$ for any $\sigma > c$.
\end{itemize}
	The subspace $\fsB_c^0(\ag{Z};\varpi) \subset \fsB_c(\ag{Z};\varpi)$ is defined by replacing (1) with
\begin{itemize}
	\item[(1')] $f(\varpi^n)=0$ for $n \ll -1$.
\end{itemize}
\end{definition}
\begin{definition}
	For any $c \in \ag{R}$ we define $\fsH_c(\ag{C};q)$ to be the space of meromorphic functions $M: \ag{C} \to \ag{C}$ satisfying
\begin{itemize}
	\item[(1)] $M(s+i\frac{2\pi}{\log q}) = M(s)$ for all $s \in \ag{C}$.
	\item[(2)] $M(s)$ is holomorphic for $\Re s > c$. 
\end{itemize}
\end{definition}
\begin{definition}
	For any $l \in [1,\infty]$ we put a system of semi-norms $B_l^{\sigma}$ for $\sigma > c$ on $\fsB_c(\ag{Z};\varpi)$ by
	$$ B_l^{\sigma}(f) = \left( \sum_{n \in \ag{Z}} q^{-n \sigma l} \norm[f(\varpi^n)]^l \right)^{\frac{1}{l}}, l \neq \infty; B_{\infty}^{\sigma}(f) = \sup_{n \in \ag{Z}} q^{-n\sigma} \norm[f(\varpi^n)]. $$
\end{definition}
\begin{definition}
	For any $l \in [1,\infty]$ we put a system of semi-norms $H_l^{\sigma}$ for $\sigma > c$ on $\fsH_c(\ag{C};q)$ by
	$$ H_l^{\sigma}(M) = \left( \int_0^{\frac{2\pi}{\log q}} \norm[M(\sigma+i\tau)]^l \frac{\log q d\tau}{2\pi} \right)^{\frac{1}{l}}, l \neq \infty; H_{\infty}^{\sigma}(M) = \max_{\Re s = \sigma} \norm[M(s)]. $$
\end{definition}
\begin{proposition}
	The topology on $\fsB_c(\ag{Z};\varpi)$ defined by $B_l^*$ does not depend on $l$. More precisely, for any $f \in \fsB_c(\ag{Z};\varpi)$ we have for $1 \leq l < \infty$ and $\epsilon > 0$ small with $\sigma - \epsilon > c$
	$$ B_l^{\sigma}(f) \ll_{\epsilon,l} B_{\infty}^{\sigma-\epsilon}(f) + B_{\infty}^{\sigma+\epsilon}(f); $$
	$$ B_{\infty}^{\sigma}(f) \leq B_l^{\sigma}(f). $$
\label{BEquiv}
\end{proposition}
\begin{proof}
	The first follows from
	$$ B_l^{\sigma}(f)^l \leq B_{\infty}^{\sigma-\epsilon}(f) \sum_{n \geq 0} q^{-n\epsilon} + B_{\infty}^{\sigma+\epsilon}(f) \sum_{n < 0} q^{n\epsilon}. $$
	The second is obvious by positivity.
\end{proof}
\begin{proposition}
	The topology on $\fsH_c(\ag{C};q)$ defined by $H_l^*$ does not depend on $l$. More precisely, for any $M \in \fsH_c(\ag{C};q)$ we have for $1 \leq l < \infty$ and $\epsilon > 0$ small with $\sigma - \epsilon > c$
	$$ H_l^{\sigma}(M) \leq H_{\infty}^{\sigma}(M); $$
	$$ H_{\infty}^{\sigma}(M) \ll_{\epsilon} H_l^{\sigma-\epsilon}(M) + H_l^{\sigma+\epsilon}(M). $$
\label{HEquiv}
\end{proposition}
\begin{proof}
	The first is obvious. For the second, let $s_0$ be any complex number with $\Re s_0 = \sigma$ and $0 < \Im s_0 < 2\pi / \log q$. Selecting the contour joining $\sigma+\epsilon$, $\sigma+\epsilon + i 2\pi / \log q$, $\sigma - \epsilon + i 2\pi / \log q$ and $\sigma-\epsilon$, we see
	$$ \log q \cdot M(s_0) = \int_0^{\frac{2\pi}{\log q}} \frac{M(\sigma+\epsilon+i\tau)}{\sigma+\epsilon+i\tau - s_0} \frac{\log q d\tau}{2\pi} - \int_0^{\frac{2\pi}{\log q}} \frac{M(\sigma-\epsilon+i\tau)}{\sigma-\epsilon+i\tau - s_0} \frac{\log q d\tau}{2\pi}. $$
	Using H\"older inequality we deduce
\begin{align*}
	\log q \cdot \norm[M(s_0)] &\leq \frac{1}{\epsilon} H_l^{\sigma+\epsilon}(M) \cdot \left( \int_0^{\frac{2\pi}{\log q}} \frac{\log q d\tau}{2\pi} \right)^{\frac{l-1}{l}} + \frac{1}{\epsilon} H_l^{\sigma-\epsilon}(M) \cdot \left( \int_0^{\frac{2\pi}{\log q}} \frac{\log q d\tau}{2\pi} \right)^{\frac{l-1}{l}} \\
	&= \frac{1}{\epsilon} H_l^{\sigma+\epsilon}(M) + \frac{1}{\epsilon} H_l^{\sigma-\epsilon}(M).
\end{align*}
	We conclude by taking $\sup$ with respect to $\Re s_0 = \sigma$.
\end{proof}
\begin{proposition}
	The two maps
	$$ \fsB_c(\ag{Z};\varpi) \to \fsH_c(\ag{C};q), f \mapsto \Mellin{f}(s) = \sum_{n \in \ag{Z}} f(\varpi^n) q^{-ns}, \text{ for } \Re s > c; $$
	$$ \fsH_c(\ag{C};q) \to \fsB_c(\ag{Z};\varpi), M \mapsto f_M(\varpi^n) = \int_0^{\frac{2\pi}{\log q}} M(\sigma+i\tau) q^{n(\sigma + i \tau)} \frac{\log q d\tau}{2\pi}, \forall \sigma > c $$
are continuous with respect to the above topologies defined by semi-norms.
\label{BHEquiv}
\end{proposition}
\begin{proof}
	The continuity follows from
	$$ H_{\infty}^{\sigma}(\Mellin{f}) \leq \sum_{n \in \ag{Z}} \norm[f(\varpi^n)] q^{-n\sigma} = B_1^{\sigma}(f); $$
	$$ B_{\infty}^{\sigma}(f_M) \leq \sup_{n \in \ag{Z}} \int_0^{\frac{2\pi}{\log q}} \norm[M(\sigma+i\tau) q^{n i \tau}] \frac{\log q d\tau}{2\pi} = H_1^{\sigma}(M). $$
\end{proof}
	Note that the abstract part of Definition \ref{FourF^1} still makes sense in the current case, i.e., for any function $f: \F \to \ag{C}$ and $\xi \in \widehat{\F^1}$ we can define
	$$ f_{\xi}(\varpi^n) = \int_{\F^1} f(\varpi^n u) \xi(u) du, n \in \ag{Z}. $$
\begin{proposition}
	For any $f \in \Sch(\F)$, $f_{\xi} \neq 0$ only for $\xi$ satisfying $\cond(\xi) \leq m(f)$, hence for only finitely many $\xi$. We have $f_{\xi} \in B_0^0(\ag{Z};\varpi)$ and
	$$ B_{\infty}^{\sigma}(f_{\xi}) \leq \Sch_{\infty}^{\sigma}(f), \forall \sigma > 0. $$
\label{SchToB}
\end{proposition}
\begin{proof}
	Obvious.
\end{proof}
\begin{lemma}
	For any $\Phi \in \Sch(\F^2)$, $s \in \ag{C}$, $\sigma > \max(0,-2\Re s)$ and $\epsilon > 0$ with $\sigma-\epsilon, \sigma+2\Re s - \epsilon >0$, there is $N=N(\epsilon, \sigma, \sigma+2\Re s) > 0$ such that with implied constant depending only on $\epsilon$
	$$ \extnorm{  \int_{\F^{\times}} \Phi(t,\frac{y}{t}) \omega^{-1}\xi^2(t)\norm[t]_{\F}^{2s} d^{\times}t } \ll_{\epsilon} q^{-ND(\Phi)+m(\Phi)} \Norm[\Phi]_{\infty} \cdot \norm[y]_{\F}^{-\sigma} 1_{y \in \vp^{2D(\Phi)}}. $$
\end{lemma}
\begin{proof}
	Writing
	$$ f(y) = \int_{\F^{\times}} \Phi(t,\frac{y}{t}) \omega^{-1}\xi^2(t)\norm[t]_{\F}^{2s} d^{\times}t, $$
we have for any $\xi_1 \in \widehat{\F^1}$ and $s_1 \in \ag{C}$ with $\Re s_1 > \max(0,-2\Re s)$
\begin{align*}
	\Mellin{f_{\xi_1}}(s_1) &= \int_{\F^{\times} \times \F^{\times}} \Phi(t,y) \omega^{-1}\xi^2\xi_1(t) \norm[t]_{\F}^{s_1+2s} \xi_1(y) \norm[y]_{\F}^{s_1} d^{\times}t d^{\times}y \\
	&= \Mellin{\Phi_{\omega^{-1}\xi^2\xi_1, \xi_1}}(s_1+2s+i\mu(\omega^{-1}\xi^2), s_1),
\end{align*}
where the second Mellin transform is the natural two dimensional one. By Proposition \ref{SchToB}, $\Phi_{\omega^{-1}\xi^2\xi_1, \xi_1} \neq 0$ only if $\cond(\omega^{-1}\xi^2\xi_1), \cond(\xi_1) \leq m(\Phi)$. In particular, the number of such $\xi_1$ is bounded by $q^{m(\Phi)}$. From the Mellin inversion for $\sigma > \max(0,-2\Re s)$
	$$ f(\varpi^n u) = \sum_{\xi_1} f_{\xi_1}(\varpi^n) \xi_1(u)^{-1} = \sum_{\xi_1} \xi_1(u)^{-1} \int_0^{\frac{2\pi}{\log q}} \Mellin{f_{\xi_1}}(\sigma+i\tau) q^{n(\sigma+i\tau)} \frac{\log q d\tau}{2\pi} $$
we can successively apply Proposition \ref{HEquiv}, \ref{BHEquiv}, \ref{BEquiv}, \ref{SchToB} and \ref{SchEquiv} to get
\begin{align*}
	\norm[y]_{\F}^{\sigma} \norm[f(y)] &\leq \sum_{\xi_1} H_1^{\sigma}(\Mellin{f_{\xi_1}}) \leq \sum_{\xi_1} H_{\infty}^{\sigma}(\Mellin{f_{\xi_1}}) \\
	&\leq \sum_{\xi_1} H_{\infty}^{\sigma+2\Re s, \sigma}(\Mellin{\Phi_{\omega^{-1}\xi^2\xi_1, \xi_1}}) \leq \sum_{\xi_1}  B_1^{\sigma+2\Re s, \sigma}(\Phi_{\omega^{-1}\xi^2\xi_1, \xi_1})  \\
	&\ll_{\epsilon} \sum_{\xi_1}  \left( B_{\infty}^{\sigma+2\Re s+\epsilon, \sigma+\epsilon} +B_{\infty}^{\sigma+2\Re s+\epsilon, \sigma-\epsilon} + B_{\infty}^{\sigma+2\Re s-\epsilon, \sigma+\epsilon} +B_{\infty}^{\sigma+2\Re s-\epsilon, \sigma-\epsilon} \right)(\Phi_{\omega^{-1}\xi^2\xi_1, \xi_1}) \\
	&\leq \left( \Sch_{\infty}^{\sigma+2\Re s+\epsilon, \sigma+\epsilon} + \Sch_{\infty}^{\sigma+2\Re s+\epsilon, \sigma-\epsilon} + \Sch_{\infty}^{\sigma+2\Re s-\epsilon, \sigma+\epsilon} + \Sch_{\infty}^{\sigma+2\Re s-\epsilon, \sigma-\epsilon} \right)(\Phi) \sum_{\xi_1}1 \\
	&\leq q^{-ND(\Phi)+m(\Phi)} \Norm[\Phi]_{\infty}.
\end{align*}
	Finally, it is obvious that $f(y) \neq 0$ implies the existence of some $t \in \F^{\times}$ such that $(t,y/t)$ lies in the support of $\Phi$ hence in $\vp^{D(\Phi)} \times \vp^{D(\Phi)}$, thus $y \in \vp^{2D(\Phi)}$.
\end{proof}
\begin{proposition}
	For any $\Phi \in \Sch(\F^2)$, let
	$$ D = \min(D(\Phi), -\cond(\psi)-\delta(\Phi)); \delta = \max(\delta(\Phi), \cond(\psi)-D(\Phi)). $$
	Then for any $\sigma > \norm[\Re s]$ and $\epsilon > 0$ with $\sigma - \epsilon > \norm[\Re s]$, there is $N=N(\epsilon, \sigma+\Re s, \sigma - \Re s)>0$ continuous in $\epsilon, \sigma\pm \Re s$ such that
	$$ \norm[W_{\Phi}(s,\xi,\omega\xi^{-1}; a(y)\kappa)] \ll_{\epsilon} q^{-(N+1)D+2\delta} \Norm[\Phi]_2 \cdot \norm[y]_{\F}^{\frac{1}{2}-\sigma} 1_{y \in \vp^{2D}}. $$
	Moreover, at an unramified place with $\cond(\psi)=\cond(\xi)=\cond(\omega\xi^{-1})=0$ and $\Phi=1_{\vo \times \vo}$ we have for any $\epsilon > 0$
	$$ \norm[W_{\Phi}(s,\xi,\omega\xi^{-1}; a(y)\kappa)] \leq \frac{2}{\epsilon \log q} \cdot \left( \sup_{x>0} \frac{x}{e^x} \right) \cdot \norm[y]_{\F}^{\frac{1}{2} - \norm[\Re s] - \epsilon} 1_{y \in \vp} + 1_{y \in \vo^{\times}}. $$
\label{LocWhiBdNA}
\end{proposition}
\begin{proof}
	The first part is a direct consequence of the previous lemma and the following inequalities deduced from Proposition \ref{DdmRel} and \ref{SchEquiv}:
	$$ D(\Four[2]{\rpR(\kappa).\Phi}) \geq \min(D(\Phi), -\cond(\psi)-\delta(\Phi)) = D; $$
	$$ m(\Four[2]{\rpR(\kappa).\Phi}) \leq \delta(\Four[2]{\rpR(\kappa).\Phi}) - D(\Four[2]{\rpR(\kappa).\Phi}) \leq \delta - D; $$
	$$ \extNorm{\Four[2]{\rpR(\kappa).\Phi}}_{\infty} \leq q^{\delta(\Four[2]{\rpR(\kappa).\Phi})} \extNorm{\Four[2]{\rpR(\kappa).\Phi}}_2 \leq q^{\delta} \Norm[\Phi]_2. $$
	The ``moreover'' part follows from a direct computation (or \cite[Theorem 4.6.5]{Bu98})
	$$ W_{\Phi}(s,\xi,\omega\xi^{-1};a(\varpi^n)) = q^{-\frac{n}{2}} \frac{\alpha^{n+1}-\beta^{n+1}}{\alpha - \beta} 1_{n \geq 0}, \text{ with } \alpha = \xi(\varpi)q^{-s}, \beta=\omega\xi^{-1}(\varpi) q^s, $$
which implies for $n \geq 1$
	$$ \norm[W_{\Phi}(s,\xi,\omega\xi^{-1};a(\varpi^n))] \leq (n+1) \norm[\varpi^n]_{\F}^{\frac{1}{2}-\norm[\Re s]} = \frac{n+1}{\epsilon n \log q} \cdot \frac{\epsilon n \log q}{q^{\epsilon n}} \cdot \norm[\varpi^n]_{\F}^{\frac{1}{2}-\norm[\Re s]-\epsilon} $$
is bounded as in the statement.
\end{proof}

		\subsubsection{Global Bound}
		
	Writing $D_v=\min(D(\Phi_v), -\cond(\psi_v)-\delta(\Phi_v))$ for $v<\infty$, it follows from Proposition \ref{LocWhiBdA} and \ref{LocWhiBdNA} that for any $\epsilon > 0$ and $N \gg 1$
	$$ \extnorm{W_{\Phi}(s,\xi,\omega\xi^{-1};na(y)\kappa)} \ll_{\epsilon,N,\Phi} \prod_{v \mid \infty} \min(\norm[y_v]_v^{\frac{1}{2}-\norm[\Re s]-\epsilon}, \norm[y_v]_v^{-N}) \prod_{v < \infty} \norm[y_v]_v^{\frac{1}{2}-\norm[\Re s]-\epsilon} 1_{y_v \in \vp_v^{2D_v}}. $$
	The bound is uniform when $\Re s$ lies in a fixed compact interval. Together with Lemma \ref{GTechEst}, we obtain
\begin{proposition}
	For any $\Phi \in \Sch(\ag{A}^2)$, $m \in \ag{N}$ and any $\sigma > \norm[\Re s]$, we have
	$$ \extnorm{\frac{\partial^m}{\partial s^m} \left( \Eis(s,\xi,\omega\xi^{-1};\Phi)(g) - \eisCst(s,\xi,\omega\xi^{-1};\Phi)(g) \right)} \leq \sum_{\alpha \in \F^{\times}} \extnorm{\frac{\partial^m}{\partial s^m} W_{\Phi}(s,\xi,\omega\xi^{-1};a(\alpha)g)} \ll_{\sigma,\Phi} \Ht(g)^{-\frac{1}{2}-\sigma}, $$
which is of rapid decay with respect to $\Ht(g)$.
\label{GlobRDEisWhi}
\end{proposition}

	\subsection{Behavior of Constant Term}
	
	Consider $f \in V_{\xi,\omega\xi^{-1}}^{\infty}$ and take its Fourier expansion into $\gp{K}$-isotypic types
	$$ f = \sum_{\vec{n}} \hat{f}_{\vec{n}} \cdot e_{\vec{n}}(\xi,\omega\xi^{-1}) $$
for some $\hat{f}_{\vec{n}} \in \ag{C}$, where $e_{\vec{n}}(\xi,\omega\xi^{-1})$ is unitary with $\gp{K}$-type parametrized by $\vec{n}$ as in \cite[Section 3.5]{Wu5}. For $\Re s \gg 1$, we have
	$$ \eisCst(s,\xi,\omega\xi^{-1};f) = f(s,\xi,\omega\xi^{-1}) + \sum_{\vec{n}} \hat{f}_{\vec{n}} \cdot \mu(s,\xi,\omega\xi^{-1};\vec{n}) e_{\vec{n}}(-s,\omega\xi^{-1},\xi), $$
where $\mu(s,\xi,\omega\xi^{-1};\vec{n})$ are the explicit coefficients of the intertwining operator on the $\gp{K}$-type $\vec{n}$ part as in \textit{loc.cit.}
\begin{proposition}
	The possible poles of the collection of meromorphic functions $\{ \mu(s,\xi,\omega\xi^{-1};\vec{n}) : \vec{n} \}$ are
\begin{itemize}
	\item The pole of $\Lambda(1-2s,\omega\xi^{-2})$ at $s=(1+i\mu(\omega\xi^{-2}))/2$ with order at most $1$, when $\omega\xi^{-2}$ is trivial on $\ag{A}^{(1)}$ and $\vec{n}=\vec{0}$. We call this pole the \emph{spherical pole}.
	\item The (both trivial and non-trivial) zeros of $L(1+2s,\omega^{-1}\xi^2)$ with order at most that of the zero.
\end{itemize}
\end{proposition}
\begin{proof}
	This can be seen either from our explicit computation of $\mu(s,\xi,\omega\xi^{-1};\vec{n})$ in \cite{Wu5}, or from (\ref{RelSec}) and
	$$ \Intw f_{\Phi}(s,\xi,\omega\xi^{-1}) = f_{\widehat{\Phi}}(-s,\omega\xi^{-1},\xi), $$
which is meromorphic with a simple pole at $s=(1+i\mu(\omega\xi^{-2}))/2$ (the other pole is cancelled by that of $f_{\Phi}(s,\xi,\omega\xi^{-1})$) by Tate's theory.
\end{proof}
\begin{lemma}
	Assume $\mu(s,\xi,\omega\xi^{-1};\vec{n}_0)$ has a pole at $s_0$ with order $n$ ($n=0$ if it is holomorphic at $s_0$) for some $\vec{n}_0$. Define
	$$ \Norm[\vec{n}] = \prod_v (\norm[n_v]+1) \text{ for } \vec{n} = (n_v)_v. $$
	Then as $s$ lying in any small compact neighborhood $K$ of $s_0$ where no other pole occur, we have
	$$ \extnorm{(s-s_0)^n \mu(s,\xi,\omega\xi^{-1};\vec{n})} \ll \Norm[\vec{n}]^N $$
for some $N$ and the implied constant depending only on $K$, $\xi$ and $\omega\xi^{-1}$ (i.e., independent of $\vec{n}$).
\end{lemma}
\begin{proof}
	This follows from the explicit computation of $\mu(s,\xi,\omega\xi^{-1};\vec{n})$ in \cite{Wu5} together with the following obvious bound
	$$ \prod_{k=1}^n \extnorm{\frac{k-s}{k+s}} \leq \prod_k \left( 1+\frac{2\norm[s]}{\norm[k+s]} \right) \leq \exp \left\{ \sum_k \frac{2\norm[s]}{\norm[k+s]} \right\} \ll (n+1)^N. $$
\end{proof}
\begin{proposition}
	Under the condition of the lemma we have
	$$ (s-s_0)^n \eisCst(s,\xi,\omega\xi^{-1};f) = (s-s_0)^n f(s,\xi,\omega\xi^{-1}) + \sum_{\vec{n}} \hat{f}_{\vec{n}} \cdot (s-s_0)^n \mu(s,\xi,\omega\xi^{-1};\vec{n}) e_{\vec{n}}(-s,\omega\xi^{-1},\xi). $$
	Consequently, for any $m \in \ag{N}$, $y \in \ag{A}^{\times}$ and $\kappa \in \gp{K}$ we can write
\begin{align*}
	&\quad \frac{d^m}{d s^m} \mid_{s=s_0} (s-s_0)^n \eisCst(s,\xi,\omega\xi^{-1};f)(a(y)\kappa) \\
	&= \frac{m!}{(m-n)!} \norm[y]_{\ag{A}}^{\frac{1}{2}+s_0} \xi(y) (\log \norm[y]_{\ag{A}})^{m-n} f(\kappa) + \sum_{k=0}^m \norm[y]_{\ag{A}}^{\frac{1}{2}-s_0} \omega\xi^{-1}(y) (\log \norm[y]_{\ag{A}})^k f_k(\kappa),
\end{align*}
where $(\log \norm[y]_{\ag{A}})^{m-n}$ is understood as $0$ if $m<n$ and where the maps
	$$ V_{\xi,\omega\xi^{-1}}^{\infty} \to V_{\omega\xi^{-1},\xi}^{\infty}, f \to f_k \text{ for } 0 \leq k \leq m $$
are $\gp{K}$-maps and continuous with respect to the Sobolev norms on $\gp{K}_{\infty}$.
\label{GlobRDEisCst}
\end{proposition}
\begin{proof}
	The lemma shows that $(s-s_0)^n \mu(s,\xi,\omega\xi^{-1};\vec{n})$ is polynomially increasing in $\vec{n}$. But $\hat{f}_{\vec{n}}$ is rapidly decreasing in $\vec{n}$ by smoothness of $f$. The sum over $\vec{n}$ is thus absolutely and uniformly convergent for $s$ lying in $K$, hence defines an analytic function in $s$. By uniqueness of analytic continuation, we get the first equation. The rest follows from the polynomial increase of $(s-s_0)^n \mu(s,\xi,\omega\xi^{-1};\vec{n})$ and Cauchy's integral formula for derivatives.
\end{proof}
\begin{remark}
	If $s_0$ is the spherical pole, then $n=1$ and for $m=0$ we have $f_k=0$ unless $k=0$. Then
	$$ f_0(\kappa) = \lim_{s \to s_0} \frac{(s-s_0)\Lambda_{\F}(2(s_0-s)}{\Lambda_{\F}(2-2(s_0-s)))} \int_{\gp{K}} f(\kappa) d\kappa = - \frac{\Lambda_{\F}^*(0)}{\Lambda_{\F}(2)} \int_{\gp{K}} f(\kappa) d\kappa. $$
\end{remark}

	\subsection{A Convergence Lemma}
	
	Let $r,N \in \ag{N}, r,N \geq 2$. It is well-known that there is an exact sequence
	$$ 1 \to \Gamma(N) \to \SL_r(\ag{Z}) \to \SL_r(\ag{Z}/N\ag{Z}) \to 1, $$
	where $\Gamma(N)$ is called the principal congruence subgroup of $\SL_r(\ag{Z})$ modulo $N$. The pre-image of the upper resp. lower triangular subgroup of $\SL_r(\ag{Z}/N\ag{Z})$ in $\SL_r(\ag{Z})$ is denoted by $\Gamma_0(N)$ resp. $\Gamma_0^-(N)$, which includes $\Gamma(N)$ as a normal subgroup. For $\vec{\alpha} \in \ag{Z}^{r-1}$ realized as a column vector, we define matrices, with $I_k$ denoting the identity matrix of rank $k$
	$$ n_r^+(\vec{\alpha}) = \begin{pmatrix}
	1 & \vec{\alpha}^T \\ \vec{0} & I_{r-1}
	\end{pmatrix}, n_r^-(\vec{\alpha}) = \begin{pmatrix}
	1 & \vec{0}^T \\ \vec{\alpha} & I_{r-1}
	\end{pmatrix}. $$
\begin{lemma}
	Any coset of $\Gamma_0(N) \backslash \SL_r(\ag{Z})$ resp. $\Gamma_0^-(N) \backslash \SL_r(\ag{Z})$ has a representative of the form $N_- N_+$ resp. $N_+N_-$, where $N_-$ resp. $N_+$ is lower resp. upper unipotent with off-diagonal entries lying in $[-\frac{N}{2}, \frac{N}{2}] \cap \ag{Z}$.
\label{CongGpLemma}
\end{lemma}
\begin{proof}
	We treat the case for $\Gamma_0(N)$. Take any $A \in \SL_r(\ag{Z})$. Let the first column of $A$ be $(a_1,\cdots,a_r)^T \in \ag{Z}^r$. Then we have $\lcd(a_1,\cdots,a_r)=1$, which implies the existence of $u_i \in \ag{Z}$ such that $\Sigma_{i=1}^r u_i a_i = 1$. In particular, we have $\lcd(u_1,a_2,\cdots,a_r)=1$. For any $k_j \in \ag{Z}, 2 \leq j \leq r$, the substitution
	$$ u_1' = u_1+\sum_{j=2}^r k_j a_j; u_j' = u_j-k_j, 2 \leq j \leq r $$
	still gives $\Sigma_{i=1}^r u_i' a_i = 1$. By an iterative application of Dirichlet's theorem on primes in arithmetic progression, we can choose $k_j$'s such that $u_1'$ is a prime number as large as we want. In particular, we can make $\lcd(u_1',N)=1$. Since $\lcd(u_1',\cdots,u_r')=1$, at least one of $u_j', 2 \leq j \leq r$, say $u_2'$, is coprime with $u_1'$. Hence $\lcd(u_1',Nu_2')=1$ and we can find $v_1,v_2 \in \ag{Z}$ such that $u_1'v_2-Nu_2'v_1=1$. The matrix
	$$ B = \begin{pmatrix} u_1' & u_2' & u_3' \cdots \\ Nv_1 & v_2 & 0 \cdots \\ \vec{0} & \vec{0} & I_{r-2} \end{pmatrix} \in \Gamma_0(N) $$
	and gives, for some $A'=\tilde{A}-\vec{\alpha} \vec{\beta_1}^T \in \SL_{r-1}(\ag{Z})$, that
	$$ BA = \begin{pmatrix} 1 & \vec{\beta_1}^T \\ \vec{\alpha} & \tilde{A} \end{pmatrix} \Rightarrow A \in \Gamma_0(N) \begin{pmatrix} 1 &  \\ \vec{\alpha} & A' \end{pmatrix} n_r^+(\vec{\beta_1}). $$
	Repeating the process on $A'$ or making an induction on $r$ we then find successively $\vec{\beta_k} \in \ag{Z}^{r-k}$ such that for some lower unipotent $N_-^1$
	$$ A \in \Gamma_0(N) N_-^1 N_+^1, N_+^1=\prod_{k=0}^{r-2} \begin{pmatrix} I_k & 0 \\ 0 & n_{r-k}^+(\vec{\beta_{k+1}}) \end{pmatrix}. $$
	Taking $N_-$ resp. $N_+$ with off-diagonal entries lying in $[-\frac{N}{2}, \frac{N}{2}] \cap \ag{Z}$ such that $N_-^1 \equiv N_- \pmod{N}$ resp. $N_+^1 \equiv N_+ \pmod{N}$, we then find $N_-^1 N_-^{-1}, N_+^1 N_+^{-1} \in \Gamma(N)$ and conclude by its normality.
\end{proof}
	Let $[\F:\ag{Q}]=r=r_1+2r_2$ where $r_1$ resp. $2r_2$ is the number of embeddings of $\F$ into real resp. complex numbers. Recall that we have a canonical map by choosing one complex embedding $\sigma_{r_1+j}, 1 \leq j \leq r_2$ in a pair conjugate to each other by complex conjugation
	$$ \sigma: \F \to \ag{R}^{r_1} \times \ag{C}^{r_2} \simeq \ag{R}^r, $$
	$$ \sigma(x) = (\sigma_1(x), \cdots, \sigma_{r_1+r_2}(x)) = (\sigma_1(x), \cdots, \sigma_{r_1}(x), \Re \sigma_{r_1+1}(x), \Im \sigma_{r_1+1}(x), \cdots, \Re \sigma_{r_1+r_2}(x), \Im \sigma_{r_1+r_2}(x)). $$
	For every fractional ideal $\idlJ$, $\sigma(\idlJ)$ is then a $\ag{Z}$-lattice of $\ag{R}^r$. For $c \gg 1$, we define a functions $f_c$ on $\ag{R}^{r_1} \times \ag{C}^{r_2}$ by
	$$ f_c(\vec{x}) = \prod_{i=1}^{r_1} \min(1, \norm[x_i]^{-c}) \prod_{j=1}^{r_2} \min(1, \norm[x_{r_1+j}]^{-2c}), \vec{x}=(x_i)_{1\leq i \leq r_1+r_2}. $$
\begin{lemma}
	Let $\idlJ \subset \vo$ be an integral ideal. We have the following two estimations for $t>0$.
\begin{itemize}
	\item[(1)] $ \sum_{\alpha \in \idlJ^{-1} - \{0\}} f_c(t\sigma(\alpha)) \ll_{\F,c} \norm[\vo/\idlJ]^{3c} t^{-c} $ if $c>r$.
	\item[(2)] $ \sum_{\alpha \in \idlJ^{-1}} f_c(t\sigma(\alpha)) \ll_{\F,c} t^{-r} \norm[\vo/\idlJ]^{-1} \left( 1+t\frac{\norm[\vo/\idlJ]^2}{\sqrt{r}} \right)^{rc} $ if $c>1$.
\end{itemize}
\label{GTechEstClassical}
\end{lemma}
\begin{proof}
	(1) On $\ag{R}^{r_1} \times \ag{C}^{r_2}$ we have a usual norm $\Norm_2$ given by
	$$ \Norm[\vec{x}]_2 = \sqrt{\sum_{i=1}^{r_1+r_2} \norm[x_i]^2}, \vec{x}=(x_i)_{1 \leq i \leq r_1+r_2}. $$
	We then easily see, essentially by comparing $f_c$ with the infinity norm, that
	$$ \sup_{\vec{x}} f_c(\vec{x}) \Norm[\vec{x}]_2^c \leq (r_1+r_2)^{\frac{c}{2}}. $$
	If $\alpha_i \in \idlJ^{-1}$ is an integral basis such that $\idlJ^{-1}=\Sigma_{i=1}^r \ag{Z} \alpha_i$, then $\sigma(\alpha_i)$ is a basis of the lattice $\sigma(\idlJ^{-1})$. We can define another norm $\Norm_{\idlJ^{-1}}$ by
	$$ \Norm[\vec{x}]_{\idlJ^{-1}}=\sqrt{\sum_{i=1}^r n_i^2}, \vec{x} = \sum_{i=1}^r n_i \sigma(\alpha_i) \in \ag{R}^r. $$
	Or equivalently, if we write $A_{\idlJ^{-1}}=(\sigma(\alpha_1), \cdots, \sigma(\alpha_r)) \in \GL_r(\ag{R})$, we have
	$$ \Norm[\vec{x}]_{\idlJ^{-1}} = \Norm[A_{\idlJ^{-1}}^{-1}\vec{x}]_2. $$
	Fix a basis $\vec{e_i}$ of $\sigma(\vo)$ and write $A_{\vo}=(\vec{e_1},\cdots,\vec{e_r})$. By elementary divisor theorem, there are $\gamma,\gamma_{\idlJ} \in \SL_r(\ag{Z})$ and some $d_i \in \ag{Z}-\{0\}, d_i \mid d_{i+1}, \norm[d_1 \cdots d_r] = \norm[\vo/\idlJ] = \norm[\idlJ^{-1} / \vo]$, such that
	$$ A_{\idlJ^{-1}}^{-1} = \gamma \diag(d_1,\cdots,d_r) \gamma_{\idlJ} A_{\vo}^{-1}. $$
	Changing $\gamma$ is equivalent to changing the choice of basis $\alpha_i$ for $\idlJ^{-1}$. Applying Lemma \ref{CongGpLemma}, we can find $\gamma$ such that
	$$ \diag(d_1,\cdots,d_r)^{-1} \gamma \diag(d_1,\cdots,d_r) \gamma_{\idlJ} = N_+ N_- $$
	where $N_+$ resp. $N_-$ is upper resp. lower unipotent with entries lying in $[-\frac{d_r}{2d_1}, \frac{d_r}{2d_1}]$. We thus get a bound of the operator norm
	$$ \Norm[A_{\idlJ^{-1}}^{-1}]_2 = \Norm[\diag(d_1,\cdots,d_r) N_+N_- A_{\vo}^{-1}]_2 \leq \norm[d_r] \left( r+\frac{r(r-1)d_r^2}{8d_1^2} \right) \Norm[A_{\vo}^{-1}]_2. $$
	We finally estimate and conclude by
\begin{align*}
	\sum_{\alpha \in \idlJ^{-1} - \{0\}} f_c(t\sigma(\alpha)) &= \sum_{\alpha \in \idlJ^{-1} - \{0\}} \Norm[t\sigma(\alpha)]_{\idlJ^{-1}}^{-c} \frac{\Norm[t\sigma(\alpha)]_{\idlJ^{-1}}^c}{\Norm[t\sigma(\alpha)]_2^c} f_c(t\sigma(\alpha)) \Norm[t\sigma(\alpha)]_2^c \\
	&\leq t^{-c} \norm[d_r]^c \left( r+\frac{r(r-1)d_r^2}{8d_1^2} \right)^c \Norm[A_{\vo}^{-1}]_2^c (r_1+r_2)^{\frac{c}{2}} \sum_{\vec{n} \in \ag{Z}^r-\{0\}} \Norm[\vec{n}]_2^{-c}.
\end{align*}

\noindent (2) If $\latL$ is a lattice in $\ag{R}^r$ with a basis given by the column vectors in a matrix $A_{\latL} \in \GL_r(\ag{R})$, we define its diameter $d(\latL)$ associated to this basis as the diameter of the fundamental parallelogram spanned by this basis, i.e.,
	$$ d(\latL) = \max_{-1 \leq \lambda_i \leq 1, 1\leq i \leq r} \Norm[A_{\latL} \vec{\lambda}]_2. $$
	Thus we have obviously $d(\latL) \leq \sqrt{r} \Norm[A_{\latL}]_2$. The same argument in the last part of (1) gives, for some choice of basis of $\idlJ^{-1}$, that
	$$ d(\sigma(\idlJ^{-1})) \leq \sqrt{r} \Norm[A_{\vo}]_2 \norm[d_1]^{-1} \left( r+\frac{r(r-1)d_r^2}{8d_1^2} \right). $$
	It is also easy to see that for $\vec{x},\vec{y}$ in the same translate of a parallelogram $\vec{x_0}+\latPlg$ of $t\sigma(\idlJ^{-1})$, we have
	$$ \frac{f_c(\vec{x})}{f_c(\vec{y})} \leq 2^{2r_2c} \left( \frac{r_1+r_2+\sqrt{r}d(t\sigma(\idlJ^{-1}))}{r} \right)^{rc}. $$
	We thus conclude by
\begin{align*}
	\Vol(\ag{R}^r/t\sigma(\idlJ^{-1})) \sum_{\alpha \in \idlJ^{-1}} f_c(t\sigma(\alpha)) &\leq 2^{2r_2c} \left( \frac{r_1+r_2+t\sqrt{r}d(\sigma(\idlJ^{-1}))}{r} \right)^{rc} \int_{\ag{R}^r} f_c(\vec{x}) d\vec{x} \\
	&\ll_{\F,c} \left( 1+t\frac{\norm[\vo/\idlJ]^2}{\sqrt{r}} \right)^{rc}.
\end{align*}
\end{proof}
	We shall write, denoting by $\vp_v$ the prime ideal corresponding to $v<\infty$, $v(\idlJ) \in \ag{N}$ such that
	$$ \idlJ = \prod_{v<\infty} \vp_v^{v(\idlJ)}. $$
	A variant of the above estimations in the adelic language is the following lemma.
\begin{lemma}
	Let $\idlJ$ be an integral ideal. Let $c_1,c_2 \in \ag{R}, c_2-c_1 > r$, we have
\begin{align*}
	&\quad \sum_{\alpha \in \F^{\times}} \prod_{v\mid \infty} \min(|\alpha y_v|_v^{-c_1},|\alpha y_v|_v^{-c_2}) \prod_{v<\infty} |\alpha y_v|_v^{-c_1} 1_{v(\alpha y_v) \geq -\cond(\psi_v)-v(\idlJ)} \\
	&\ll_{\F,c_2-c_1} \min \left(|y|_{\ag{A}}^{-c_1-1} \norm[\vo/\idlJ]^{-1} \left( 1+|y|_{\ag{A}}^{\frac{1}{r}} \frac{\norm[\vo/\idlJ]^2}{\sqrt{r}} \right)^{r(c_2-c_1)}, |y|_{\ag{A}}^{-c_1-\frac{c_2-c_1}{r}} \norm[\vo/\idlJ]^{3(c_2-c_1)} \right)
\end{align*}
	where $y=(y_v)_v \in \ag{A}^{\times}$, $r=[\F:\ag{Q}]$.
\label{GTechEst}
\end{lemma}
\begin{proof}
	The following exact sequence is split with a splitting section $s: \ag{R}_+ \to \ag{A}^{\times}$
	$$ 1 \to \ag{A}^{(1)} \to \ag{A}^{\times} \xrightarrow{|\cdot|_{\ag{A}}} \ag{R}_+ \to 1 $$
	such that the image of $s$ is contained in $\ag{A}_{\infty}^{\times} = \prod_{v \mid \infty} \F_v^{\times}$. For any $t \in \ag{R}_+$, we write $t^+=s(t)$ with $t_v^+=t^{1/r} \in \ag{R}_+ \subset \F_v$, such that
	$$ |t^+|_{\ag{A}} = \prod_{v \mid \infty} |t_v^+|_v = t. $$
	We may apply the compactness of $\F^{\times} \backslash \ag{A}^{(1)}$, or proceed alternatively in the following more classical way. Let $\gCl(\F)$ be the class group of $\F$ and choose an integral ideal $\tau$ in each class $[\tau] \in \gCl(\F)$. Let $\delta_{\tau} \in \ag{A}_{\fin}^{\times}$ be a representative of $\tau$ in the group of ideles. Since $\gCl(\F) \simeq \F^{\times} \backslash \ag{A}_{\fin}^{\times} / \widehat{\vo}^{\times}$, there is a unique $\tau$, and some $\beta=\beta_0 u \in \F^{\times}$ with freely chosen $u \in \vo^{\times}$ such that $ y_{\fin} \in \beta \delta_{\tau} \widehat{\vo}^{\times}$. Let $t=\norm[y]_{\ag{A}}$. We can write $y=y_{\infty} y_{\fin} = t^+ y_{\infty}' y_{\fin}$, and find
	$$ \prod_{v \mid \infty} \norm[\beta^{-1}y_v(t_v^+)^{-1}]_v = \prod_{v \mid \infty} \norm[u^{-1}\beta_0^{-1}y_v' ]_v = t^{-1}\norm[\beta^{-1}y]_{\ag{A}} \norm[\delta_{\tau}]_{\ag{A}}^{-1} = \norm[\vo/\tau]. $$
	By Dirichlet's unit theorem, we can adjust $u \in \vo^{\times}$ so that each $\norm[u^{-1}\beta_0^{-1}y_v' ]_v$ remains in a compact set in $\ag{R}_+$ depending only on $\F$. We thus find
\begin{align*}
	&\quad \sum_{\alpha \in \F^{\times}} \prod_{v\mid \infty} \min(|\alpha y_v|_v^{-c_1},|\alpha y_v|_v^{-c_2}) \prod_{v<\infty} |\alpha y_v|_v^{-c_1} 1_{v(\alpha y_v) \geq -\cond(\psi_v)-v(\idlJ)} \\
	&= \sum_{\alpha \in \F^{\times}} \prod_{v\mid \infty} \min(|\alpha \beta^{-1} y_v|_v^{-c_1},|\alpha \beta^{-1} y_v|_v^{-c_2}) \prod_{v<\infty} |\alpha \delta_{\tau}|_v^{-c_1} 1_{v(\alpha) \geq -\cond(\psi_v)-v(\idlJ \tau)} \\
	&= t^{-c_1} \sum_{\alpha \in (\Dif_{\F} \idlJ \tau)^{-1} - \{0\}} \prod_{v\mid \infty} \min(1,\norm[\alpha t_v^+]_v^{-(c_2-c_1)} \cdot \norm[u^{-1}\beta_0^{-1}y_v']_v^{c_1-c_2}) \\
	&\ll_{\F,c_2-c_1} t^{-c_1} \sum_{\alpha \in (\Dif_{\F} \idlJ \tau)^{-1} - \{0\}} f_{c_2-c_1}(t^{\frac{1}{r}} \sigma(\alpha)).
\end{align*}
	We then apply Lemma \ref{GTechEstClassical} to conclude.
\end{proof}

\section*{Acknowledgement}

	The preparation of the paper scattered during the stays of the author's in FIM at ETHZ, YMSC at Tsinghua University and Alfr\'ed Renyi Institute in Hungary supported by the MTA R\'enyi Int\`ezet Lend\"ulet Automorphic Research Group. The author would like to thank all three institutes for their hospitality.

\bibliographystyle{acm}

\bibliography{mathbib}

\begin{thebibliography}{1}

\bibitem{Bu98}
{\sc Bump, D.}
\newblock {\em Automorphic Forms and Representations}.
\newblock No.~55 in Cambridge Studies in Advanced Mathematics. Cambridge
  University Press, 1998.

\bibitem{GJ79}
{\sc Gelbart, S.~S., and Jacquet, H.}
\newblock Forms of {GL}(2) from the analytic point of view.
\newblock In {\em Proceedings of Symposia in Pure Mathematics\/} (1979),
  vol.~33, pp.~213--251.

\bibitem{J04}
{\sc Jacquet, H.}
\newblock Integral representation of {W}hittaker functions.
\newblock In {\em Contributions to Automorphic Forms, Geometry \& Number
  Theory}, H.~Hida, D.~Ramakrishnan, and F.~Shahidi, Eds. The Johns Hopkins
  University Press, 2004, ch.~15, pp.~373--419.

\bibitem{JS90}
{\sc Jacquet, H., and Shalika, J.}
\newblock {R}ankin-{S}elberg convolutions: {A}rchimedean theory.
\newblock In {\em Festschrift in honor of I. I. Piatetski-Shapiro on the
  occasion of his sixtieth birthday\/} (Jerusalem, 1990), vol.~Part I, Weizmann
  Science Press, pp.~125--207.

\bibitem{MV10}
{\sc Michel, P., and Venkatesh, A.}
\newblock The subconvexity problem for {$GL_2$}.
\newblock {\em Publications math\'{e}matiques de l'IH\'{E}S 111}, 1 (June
  2010), 171--271.

\bibitem{Wu5}
{\sc Wu, H.}
\newblock A note on spectral analysis in automorphic representation theory for
  {${\rm GL}_2$: \Rmnum{1}}.
\newblock {\em accepted by IJNT, arXiv: 1609.08742\/}.

\bibitem{Za82}
{\sc Zagier, D.}
\newblock The {R}ankin-{S}elberg method for automorphic functions which are not
  of rapid decay.
\newblock {\em Journal of the Faculty of Science 28\/} (1982), 415--438.

\end{thebibliography}

\address{\quad \\ Han WU \\ 302, Alfr\'ed-Renyi Institute \\ Re\'altanoda utca 13-15 \\ 1053, Budapest \\ Hungary \\ wuhan1121@yahoo.com}

\end{document}